%% file: Unknotting-V7.tex
\documentclass[11pt]{amsart}
\usepackage{hyperref,flafter,ifsym,wasysym,color}
\usepackage{epsfig}
\usepackage[english]{babel}
\usepackage{amsfonts}
\usepackage{amssymb}
\usepackage{amsmath}
\usepackage{amsthm}
\usepackage{latexsym}
\usepackage{amscd}
\usepackage{epsf}
\usepackage{caption}
\usepackage[lmargin=1in,rmargin=1in,tmargin=0.8in,bmargin=0.8in]
{geometry}
\usepackage{enumitem}
\usepackage{todonotes}
\input{diagrams}
\input{commands}
\begin{document}
\title{Knot Floer homology and the 
unknotting number}%
\author{Akram Alishahi}
\thanks{AA was supported by NSF Grants DMS-1505798 and DMS-1811210.}

\address{Department of Mathematics, Columbia University, 
New York, NY 10027}
\email{alishahi@math.columbia.edu}
\author{Eaman Eftekhary}%
\address{School of Mathematics, Institute for Research in 
Fundamental Sciences (IPM), P. O. Box 19395-5746, Tehran, Iran}%
\email{eaman@ipm.ir}
\begin{abstract}
Given a knot $K\subset S^3$, let $u^-(K)$ (respectively, $u^+(K)$) 
denote the minimum number of negative (respectively, positive) 
crossing changes  among all unknotting sequences for $K$. We use 
knot Floer homology to construct the invariants $\depth^-(K),
\depth^+(K)$ and $\depth(K)$, which give lower bounds on $u^-(K),
u^+(K)$ and the unknotting number $u(K)$, respectively. The invariant 
$\depth(K)$ only vanishes for the unknot, and satisfies 
$\depth(K)\geq \nu^-(K)$, while the difference 
$\depth(K)-\nu^-(K)$ can be arbitrarily large.
We also present several applications towards bounding the 
unknotting number, the alteration number and the Gordian distance.  
\end{abstract}
\maketitle
\section{Introduction}
Given a knot $K\subset S^3$, 
by an \emph{unknotting sequence} for $K$ we mean a sequence of 
crossing changes  for $K$ which results in  the unknot. 
The minimum length of an unknotting sequence for $K$ is 
called the {\emph{unknotting number}} of $K$ and is denoted 
by $u(K)$. Let $u^-(K)$ denote the minimum number of
negative crossing changes (i.e. changes of a  
positive crossing  to a  negative crossing) among 
all unknotting sequences for $K$ and  $u^+(K)$ denote the 
minimum number of  positive crossing changes 
among all such sequences. It is then clear that 
$u(K)\geq u^+(K)+u^-(K)$, while the equality is not necessarily 
satisfied. The unknotting number is one of the simplest,
yet most mysterious and intractable invariants of knots in 
$S^3$.   The answer to several simple questions about 
the unknotting number is still not known. In particular, the 
the following question is widely open.

\begin{quest}\label{quest:additivity}
If $K$ and $L$ are knots in $S^3$, is it true that 
$u(K\#L)=u(K)+u(L)$? How about the (weaker) inequality 
$u(K\#L)\geq \max\{u(K),u(L)\}$?
\end{quest}

Scharlemann proved that composite knots have unknotting number 
at least $2$ \cite{Schar}. 
However, no matter how large $u(K)$ and $u(L)$ are, it is not 
known in general whether $u(K\#L)\geq 3$ \cite{Lackenby}.

Another example is Milnor's conjecture on the unknotting number 
of the torus knot $T_{p,q}$, which remained open for a long time, 
until a proof was given by Kronheimer and Mrowka using 
gauge theory \cite{KM-Milnor-conj}. 
Ozsv\'ath and Szab\'o reproved
it using their invariant $\tau(K)$ \cite{OS-tau} and
Rasmussen gave a purely combinatorial proof  by introducing 
his invariant $s(K)$ \cite{Rasmussen-s}. Both 
$|\tau(K)|$ and $|s(K)|/2$, as well as classical lower bounds for 
the unknotting number coming from Levine-Tristram signatures
\cite{J-Levine,Tristram}, are in fact lower bounds for 
the $4$-ball genus $g_4(K)$. Since $g_4(K)\leq u(K)$, they 
also give lower bounds for the unknotting number.
Nevertheless, lower bounds for $u(K)$ constructed 
by bounding the $4$-ball genus fail to give effective data 
 for many classes of knots. In particular, if 
$-{K}$ denotes the mirror image of the knot $K$, the knot
$L=K\#-{K}$ is always slice and $\tau(L)=s(L)=0$. 
It is thus interesting to construct lower bounds for 
$u(K)$, which do not come from bounds on $g_4(K)$.
In this paper, we use knot Floer homology to construct the 
invariants $\depth^-(K),\depth^+(K)=\depth^-(-{K})$ 
and $\depth(K)$ associated with a knot $K\subset S^3$  
and prove the following theorem.

\begin{thm}\label{thm:intro}
For every knot $K\subset S^3$ we have 
\begin{itemize}
\item $\depth^+(K)\leq u^+(K)$, $\depth^-(K)\leq u^-(K)$ and 
$\depth(K)\leq u(K)$.
\item $\depth^-(K)\geq \nu^-(K)\geq \tau(K)$ and 
$\depth^+(K)\geq \nu^-(-{K})\geq -\tau(K)$.  Therefore, for 
every $0\le t\le 1$ we have $t\depth^+(K)\geq \Upsilon_K(t)
\geq -t\depth^-(K)$.
\item  $\depth(K)\geq \widehat{\Tt}(K)$ where $\widehat{\Tt}(K)$ 
is the maximum order of $U$-torsion in $\mathrm{HFK}^-(K)$.
\end{itemize}
\end{thm}
Unlike most other lower bounds for the unknotting number, the torsion 
invariant $\widehat{\Tt}$ resists the connected sum operation.
\begin{cor}\label{cor-2:intro}
If $K$ and $K'$ are knots in $S^3$ then
\[u(K\#K')\geq \widehat{\Tt}(K\#K')\ge
\max\{\widehat{\Tt}(K),\widehat{\Tt}(K')\}.\]
In particular, for the torus knot $T_{p,q}$ with 
$0<p<q$, $\widehat{\Tt}(T_{p,q})=p-1$ and for every knot 
$K\subset S^3$
\[u(K\# T_{p,q})\geq p-1.\]
\end{cor}
Therefore, for every coprime $0<p<q$, 
$\depth(-T_{p,q}\#T_{p,q})\ge p-1$, while the lower bounds 
$\nu^-, |\tau|$ and $|s|/2$ vanish, because $-T_{p,q}\#T_{p,q}$ 
is slice.

Theorem~\ref{thm:intro} naturally reproves the following corollary.
\begin{cor}\label{cor-1:intro}
For every knot $K\subset S^3$, $\nu^-(K)$ is a lower bounds
for $u^-(K)$, while $\nu^-(-K)$ is a lower bound for 
$u^+(K)$.  In particular, $u^-(T_{p,q})=(p-1)(q-1)/2$. 
\end{cor}

Associated with a knot $K\subset S^3$, one can construct a 
Heegaard Floer chain complex $\CFT(K)$, which is freely generated 
over $\Abb=\Fbb[\la,\law]$ by the intersection points associated 
with a Heegaard diagram for $K$. $\CFT(K)$ is equipped with 
differential $d$, which is an $\Abb$-homomorphism defined by 
counting holomorphic disks \cite{AE-1}. 
Let $\Hbb(K)$ denote the homology of $(\CFT(K),d)$, which is 
again a module over $\Abb$. Let 
$\Tbb(K)$ denote the torsion submodule of $\Hbb(K)$, i.e.
$\Tbb(K)$ consists of  $x\in \Hbb(K)$ such that there exists
a non-zero $a\in\Abb$ with $a\cdot x=0$.  Then, $\Hbb(K)$ sits 
in a short exact sequence 
\begin{displaymath}
\begin{diagram}
0&\rTo& \Tbb(K)&\rTo&\Hbb(K)&\rTo&\Abb(K)&\rTo&0,
\end{diagram}
\end{displaymath}
where the torsion free part $\Abb(K)$ of the homology 
is isomorphic to an ideal in $\Abb$. Specifically, for every
knot $K$, there is an {\emph{ideal sequence}} 
$\imath(K)=(i_0=0<i_1<\cdots<i_n=\nu^-(K))$ of some length 
$n=n(K)$ and a canonical identification  
\begin{align*}
\Abb(K)=\langle \la^{i_k}\law^{i_{n-k}}\ |\ k=0,1,\ldots,n
\rangle_\Abb\leq \Abb.
\end{align*}
We define $\Tt(K)$ as the 
smallest integer $m$ such that $\law^m$ acts trivially on 
$\Tbb(K)$ (i.e. maps $\Tbb(K)$ to zero). For the unknot $U$, 
we have $\Tbb(U)=0$ and $\Abb(U)=\Abb$.

If $K'$ is obtained from $K$ by a sequence of $m$ 
negative crossing changes and $n$ positive crossing changes,
we use the cobordism maps constructed in \cite{AE-2} to show 
that $\law^n\Abb(K)\subset \Abb(K')$ and 
$\law^m\Abb(K')\subset \Abb(K)$, while 
$\law^{m+n}\Tbb(K)$ may be embedded in $\Tbb(K')$.
This observation implies, in particular, that $\nu^-(K)$ is a 
lower bound for $u^-(K)$ and that $\Tt(K)$ is lower bound for 
$u(K)$. 

The above construction also gives lower bounds on the {\emph{Gordian
distance}} $u(K,K')$ from a knot $K$ to another knot $K'$, i.e.
the minimum number of crossing changes required to get from $K$ 
to $K'$. In particular, we give the following three lower bounds on 
the {\emph{alteration number}} $\alt(K)$, defined as the least 
Gordian distance of an alternating knot from $K$.

\begin{prop}\label{prop:intro-alternation-for-K}
The alternation number $\alt(K)$ of a knot $K\subset S^3$
satisfies the inequalities
\[\alt(K)\geq \nu^-(K)-\md(K),\quad\alt(K)\geq \widehat{\Tt}(K)-1
\quad\text{and}\quad
\alt(K)\geq\min\{\Tt(K)-1,|\nu^-(K)|\},\]
where $\md(K)$ is the minimum degree of a monomial in $\Abb(K)$.
In particular, it follows that 
\[\alt(T_{p,pn+1})\geq n\left\lfloor\frac{(p-1)^2}{4}\right\rfloor.\]
\end{prop}

A similar strategy is used by the first author in \cite{A-BN}
to construct lower bounds for the unknotting number from 
Khovanov homology. The resulting invariants are used in 
\cite{AD-KM} to prove the {\emph{knight move conjecture}}
for knots with unknotting number at most $2$.

In Section~\ref{sec:construction}
we study the cobordism maps induced on knot chain complexes 
associated with a crossing change. These cobordism maps are used in 
Section~\ref{sec:depth} to construct lower bounds on the Gordian 
distance of knots, while simpler obstructions to the unknotting 
are extracted from these lower bounds in Section~\ref{sec:torsion}.
We discuss several examples and applications in 
Section~\ref{sec:examples}.\\

{\bf{Acknowledgements.}} The authors would like to thank 
Jennifer Hom, Robert Lipshitz and Iman Setayesh
for helpful discussions and suggestions. 

\section{Changing the crossings in knot diagrams}\label{sec:construction} 
By a {\emph{crossing change}} for an oriented link $L\subset Y$ 
we mean replacing a ball in $Y$ in which $L$ looks like a 
positive crossing to the ball in which $L$  looks like a 
negative crossing (a negative crossing change), or the reverse 
of the above operation (a positive crossing change). 
Figure~\ref{fig:crossing-change} illustrates how a band surgery 
on $L$ can be used to do any of the following two changes (or 
the reverse of it):
\begin{itemize}
\item A negative crossing change and adding 
a positively oriented meridian for $L$ as a new link component.
\item A positive crossing change and adding 
a negatively oriented meridian for $L$ as a new link component.
\end{itemize}  
Let us assume that  $K'$ is obtained from $K$ by a negative 
crossing change and that $L$ is obtained from $K'$ by adding 
a positively oriented meridian. As illustrated in 
Figure~\ref{fig:crossing-change}, one may then place a pair of 
markings $p_1,p_2$ on $K$, and distinguish a band 
$\Ibb$ with endpoints on $K\setminus\{p_1,p_2\}$, such that the band 
surgery on $\Ibb$ gives $L$, while $p_1$ lands on $K'$ and 
$p_2$ lands on the positively oriented meridian.

Associated with the pointed link $(K,p_1,p_2)$, we may construct 
a tangle (equivalently, a sutured manifold) as follows. Fix an 
orientation on $K$ and consider two disjoint small arc on $K$ 
which contains $p_1$ and $p_2$, respectively. Remove a small 
ball around each one of the four ends of these arcs to obtain 
a $3$-manifold $M$ with $4$ sphere boundary components. Using 
the orientation on $K$ we may orient these spheres so that two 
of them form $\partial^+M$ and the other two form $\partial^-M$, 
see the lower part of Figure~\ref{fig:crossing-change}.
Let $T_1$ and $T_2$ denote the remaining part of the two arcs 
around $p_1$ and $p_2$, respectively, which are now strands in 
$M$ connecting $\partial^+M$ to $\partial^-M$. The complement 
of the two arcs in $K$ gives two other strands $T_3$ and $T_4$
which connect $\partial^+M$ to $\partial^-M$. The $3$-manifold
$M$ and the strands $T_1,T_2,T_3$ and $T_4$ then form a 
tangle associated with $(K,p_1,p_2)$ (see~\cite{AE-2}).
Correspondingly, we also obtain a sutured manifold, which 
is constructed by removing a solid cylinder around each one of 
the strands and considering the boundary of these $4$ solid 
cylinders as the set of sutures on the resulting $3$-manifold.
The construction of authors in \cite{AE-1}, as well as the 
special case considered in \cite[Subsection 8.2]{AE-1},
may be used to associate a chain complex
$\CFT(K,p_1,p_2)$  with this tangle (or sutured 
manifold), which is a module over $\Abb'=\Fbb[\la,\lav,\law]$.
The variables $\la$ and $\lav$ are associated with the strands 
$T_1$ and $T_2$ (equivalently, with $p_1$ and $p_2$), while 
the variable $\law$ is associated with $T_3$ and $T_4$ 
(equivalently, with $K\setminus\{p_1,p_2\}$). Similarly, we 
can associate a chain complex $\CFT(L,p_1,p_2)$ with the 
pointed link $(L,p_1,p_2)$, which is again a module over 
$\Abb'$. The generators of the two complexes all correspond 
to the unique $\SpinC$ structure $\spinc_0$ on $S^3$, which 
will be dropped from the notation.
 
Associated with the framed arc $\Ibb$, the construction of 
\cite{AE-2} gives the $\Abb'$-cobordism maps
\begin{align*}
\gmap^+:\CFT(K,p_1,p_2)\ra \CFT(L,p_1,p_2)
\quad \text{and}\quad
\gmap^-:\CFT(L,p_1,p_2)\ra \CFT(K,p_1,p_2).
\end{align*} 

\begin{figure}
\def\svgwidth{11cm}
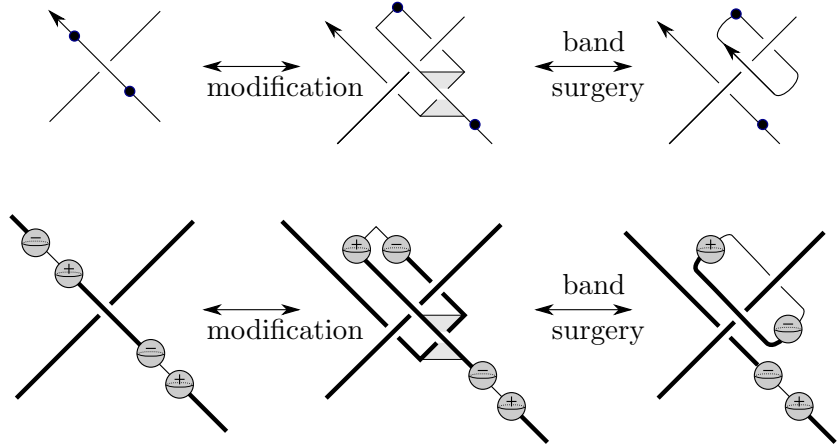
\caption{
We may change a crossing in the expense of adding a meridian. 
The meridian can be positively or negatively oriented depending 
on whether the initial crossing is negative or positive, 
respectively. 
}\label{fig:crossing-change}
\end{figure} 

\begin{lem}\label{lem:reversing-arcs}
With the above notation fixed, the maps 
\[\gmap^+\circ \gmap^-:\CFT(L,p_1,p_2)\ra \CFT(L,p_1,p_2)
\quad\text{and}\quad 
\gmap^-\circ \gmap^+:\CFT(K,p_1,p_2)\ra \CFT(K,p_1,p_2)
\]
are both multiplication by $\law$, up to chain homotopy.
\end{lem} 
\begin{proof}
For defining $\gmap^+$ we may use a triple 
\[(\Sig,\alphas,\betas,\gammas,\z=\{z_1,z_2,z_3,z_4\})\] 
subordinate to the framed arc $\Ibb$, where $z_i$ corresponds
to the strand $T_i$. The corresponding $\Abb'$-coloring maps
$z_1$  to $\la$ and $z_2$  to $\lav$, while $z_3$ and $z_4$ 
are mapped to $\law$. If $\deltas$ is obtained by a 
small Hamiltonian isotopy from $\betas$ which do not cross 
$\z$, then $(\Sig,\alphas,\gammas,\deltas,\z)$ is 
subordinate to $\ovl{\Ibb}$, the reverse band surgery. 
Associated with the Heegaard quadruple
$H=(\Sig,\alphas,\betas,\gammas,\deltas,\z)$
we obtain: 
\begin{itemize}
\item the top generators $\Theta_{\beta\gamma}, 
\Theta_{\gamma\delta}$ and $\Theta_{\beta\delta}$ in 
$\Tb\cap\Tc, \Tc\cap\Td$ and $\Tb\cap \Td$, respectively,
\item the triangle maps $\fmap_{\alpha\beta\gamma}$,
$\fmap_{\alpha\gamma\delta}$,
$\fmap_{\alpha\beta\delta}$ and $\fmap_{\beta\gamma\delta}$,
which are associated with the triples $(\alphas,\betas,\gammas)$,
$(\alphas,\gammas,\deltas)$, $(\alphas,\betas,\deltas)$
and $(\betas,\gammas,\deltas)$, respectively, and the induced 
maps $\gmap^+=\fmap_{\alpha\beta\gamma}
(-\otimes\Theta_{\beta\gamma})$ and $\gmap^-
=\fmap_{\alpha\gamma\delta}(-\otimes\Theta_{\gamma\delta})$,
\item and the holomorphic square map $\Sqmap$ which satisfies 
\[d\circ\Sqmap+\Sqmap\circ d=
\gmap^+\circ\gmap^-+
\fmap_{\alpha\beta\delta}(-\otimes 
\fmap_{\beta\gamma\delta}(\Theta_{\beta\gamma}
\otimes\Theta_{\gamma\delta})).\]
\end{itemize}
The position of the curves in $\betas\cup\gammas\cup\deltas$, 
which is basically illustrated in 
Figure~\ref{fig:reversing-arcs}, implies that 
\[\fmap_{\beta\gamma\delta}(\Theta_{\beta\gamma}
\otimes\Theta_{\gamma\delta})=\law\Theta_{\beta\delta}.\]
Since $\fmap_{\alpha\beta\delta}(-\otimes\Theta_{\beta\delta})$ 
gives a map chain homotopic to the identity on 
$\CFT(L,p_1,p_2)$, the above observation completes the proof for 
the composition $\gmap^+\circ\gmap^-$. A similar argument implies 
that $\gmap^-\circ\gmap^+$ is chain homotopic to multiplication 
by $\law$.
\end{proof}

Removing $p_2$ from $K$, we obtain a knot with a single marked 
point on it. Correspondingly, we find a tangle with two strands 
and the  standard knot chain complex $\CFT(K)$ for 
$K$, which is a module over $\Abb=\Fbb[\la,\law]$. Similarly, 
there is a single marked point on $K'$, and associated with it 
we obtain the chain complex $\CFT(K')$, which is again an 
$\Abb$-module. There are chain homotopy equivalences
\begin{align*}
&\CFT(K,p_1,p_2)\simeq\left(\CFT(K)\otimes_{\Abb}
\Abb'\xrightarrow{\la+\lav}\CFT(K)\otimes_{\Abb}\Abb'\right)
\quad\text{and}\\
&\CFT(L,p_1,p_2)\simeq\left(\left(\CFT(K')\oplus\CFT(K')\right)
\otimes_{\Abb} \Abb'\xrightarrow{\begin{bmatrix}0&\law\\ \la+\lav
&\la \end{bmatrix}} \left(\CFT(K')\oplus\CFT(K')\right)
\otimes_{\Abb} \Abb'\right),
\end{align*}
where the latter is deduced from the identification 
$L=K'\#(\text{Hopf link})$, and the chain complex 
$(C\xrightarrow{f}C')$ is defined as the mapping cylinder of 
the chain map $f:C\ra C'$ between two chain complexes.
Corresponding to the above chain homotopy equivalences, we may 
present $\gmap^+$ and $\gmap^-$ as $4\times 2$ and $2\times 4$ 
matrices 
$(\gmap^+_{ij})_{ij}$ and 
$(\gmap^-_{ji})_{ji}$, where 
\begin{align*}
&\gmap^+_{ij}:\CFT(K)\otimes_{\Abb}\Abb'
\ra \CFT(K')\otimes_{\Abb} \Abb'\quad\text{and}
&\gmap^-_{ji}:\CFT(K')\otimes_{\Abb} \Abb'
\ra \CFT(K)\otimes_{\Abb}\Abb'.
\end{align*}

\begin{figure}
\def\svgwidth{5.5cm}
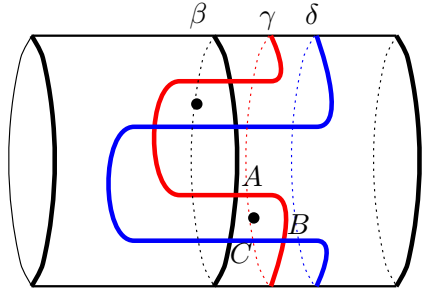
\caption{
If $\deltas$ is obtained from $\betas$ by a Hamiltonian 
isotopy supported away from the marked points, the domain of the 
distinguished triangle class in $(\Sig,\betas,\gammas,\deltas,\z)$
connecting $\Theta_{\beta\gamma},
\Theta_{\gamma\delta}$ and $\Theta_{\beta\delta}$ will 
contain one of the marked points corresponding to the strands 
connected by $\Farc$. The intersection of the domain of the 
 triangle with the surface is the small triangle 
connecting $A,B$ and $C$.
}\label{fig:reversing-arcs}
\end{figure}

Let us denote $\la+\lav$ by $\sigma$ and regard $\Abb'$ as 
$\Abb[\sigma]$. For each $1\le i,j\le 2$, we decompose
\[\gmap^+_{ij}=\gmap^+_{ij,0}+\sigma \hmap_{ij}^+\quad\text{and}
\quad \gmap^-_{ji}=\gmap^-_{ji,0}+\sigma\hmap_{ji}^-,\]
where the maps $\gmap^\pm_{ij,0}$ do not use the variable $\sigma$.
We will find chain homotopies such that
\begin{equation}\label{eq:simchainmap}
\gmap^+\simeq\begin{bmatrix}
\gmap_{11,0}^+&0\\
\gmap_{21,0}^+&0\\
\gmap_{31,0}^+&0\\
\gmap_{41,0}^+&\gmap_{11,0}^+
\end{bmatrix}\quad\text{and}\quad
\gmap^-\simeq\begin{bmatrix}
\gmap_{24,0}^-&0&0&0\\
\gmap_{21,0}^-&\gmap_{22,0}^-&\gmap_{23,0}^-&\gmap_{24,0}^-
\end{bmatrix}.
\end{equation}
First, we deduce from $\gmap^+$ and $\gmap^-$ being chain maps that 
\[\begin{split}
&\sigma. \gmap_{12}^+=\gmap_{11}^+\circ d+d\circ \gmap_{11}^+\\
&\sigma.\gmap_{22}^+=\gmap_{21}^+\circ d+d\circ\gmap_{21}^+\\
&\sigma.\gmap_{32}^+=\law.\gmap_{21}^++\gmap_{31}^+\circ d
+d\circ\gmap_{31}^+\\
&\sigma.\gmap_{42}^+=\sigma.\gmap_{11}^++\la.\gmap_{21}^+
+\gmap_{41}^+\circ d+d\circ\gmap_{41}^+
\end{split}\quad\text{and}\quad\quad 
\begin{split}
&\sigma.\gmap_{11}^-=\sigma.\gmap_{24}^-+\gmap_{21}^-\circ d
+d\circ\gmap_{21}^-\\
&\sigma.\gmap_{12}^-=\law.\gmap_{23}^-+\la.\gmap_{24}^-
+\gmap_{22}^-\circ d+d\circ\gmap_{22}^-\\
&\sigma.\gmap_{13}^-=\gmap_{23}^-\circ d+d\circ\gmap_{23}^-\\
&\sigma.\gmap_{14}^-=\gmap_{24}^-\circ d+d\circ\gmap_{24}^-
\end{split}
\]
The differentials $d$ of the complexes do not use the variable 
$\sigma$, so the above equations imply
\[\begin{split}
&\gmap_{12}^+=\hmap_{11}^+\circ d+d\circ \hmap_{11}^+\\
&\gmap_{22}^+=\hmap_{21}^+\circ d+d\circ\hmap_{21}^+\\
&\gmap_{32}^+=\law.\hmap_{21}^++\hmap_{31}^+\circ d
+d\circ\hmap_{31}^+\\
&\gmap_{42}^+=\gmap_{11}^++\la.\hmap_{21}^++\hmap_{41}^
+\circ d+d\circ\hmap_{41}^+
\end{split}\quad\text{and}\quad\quad 
\begin{split}
&\gmap_{11}^-=\gmap_{24}^-+\hmap_{21}^-\circ d+d\circ\hmap_{21}^-\\
&\gmap_{12}^-=\law.\hmap_{23}^-+\la.\hmap_{24}^-+\hmap_{22}^-\circ d
+d\circ\hmap_{22}^-\\
&\gmap_{13}^-=\hmap_{23}^-\circ d+d\circ\hmap_{23}^-\\
&\gmap_{14}^-=\hmap_{24}^-\circ d+d\circ\hmap_{24}^-
\end{split}
\]
Then, it is easy to check that 
\[
H^+=\begin{bmatrix}
0&\hmap_{11}^+\\
0&\hmap_{21}^+\\
0&\hmap_{31}^+\\
0&\hmap_{41}^+
\end{bmatrix}\quad\text{and}\quad
H^-=\begin{bmatrix}
\hmap_{21}^-&\hmap_{22}^-&\hmap_{23}^-&\hmap_{24}^-\\
0&0&0&0
\end{bmatrix}
\]
are the chain homotopies for $\gmap^+$ and $\gmap^-$ which result in 
Equation \ref{eq:simchainmap}, respectively. Abusing the notation we 
keep denoting the new matrixes by $\gmap^-=(\gmap_{ij}^-)$ and 
$\gmap^+=(\gmap_{ij}^+)$.

We now set $\sigma=0$, or equivalently $\lav=\la$. Then, 
$\gmap_{11}^+$ and $\gmap_{11}^-$ induce chain maps
\[\gmap_{11}^+:\CFT(K)\to\CFT(K')\quad\text{and}\quad
\gmap_{11}^-:\CFT(K')\to \CFT(K),\]
and we define $\fmap^+=\gmap^+_{11}$ and $\fmap^-=\gmap^-_{11}$.
Note that $(\gmap^+\circ\gmap^-)_{11}=\gmap^+_{11}\circ\gmap^-_{11}$ 
and $(\gmap^-\circ\gmap^+)_{11}=\gmap^-_{11}\circ\gmap^+_{11}$. 
So, both $\fmap^+\circ\fmap^-$ and $\fmap^-\circ\fmap^+$ are chain 
homotopic to multiplication by $\law$.

If $K'$ is obtained from $K$ by a positive crossing change,
a similar argument may be used to arrive at the same conclusion. 
The above discussion implies the following theorem.

\begin{thm}\label{thm:cobordisim-maps}
If $K'\subset S^3$ is obtained from $K\subset S^3$ by a 
crossing change, there exist chain maps
\begin{align*}
\fmap^+:\CFT(K)\ra \CFT(K')
\quad\text{and}\quad
\fmap^-:\CFT(K')\ra \CFT(K')
\end{align*}
such that $\fmap^+\circ\fmap^-$ and $\fmap^-\circ\fmap^+$ 
are chain homotopic to multiplication by $\law$.\end{thm}

Given a knot $K\subset S^3$, the knot Floer chain complex $\CFT(K)$ 
(which is generated over $\Abb=\Fbb[\la,\law]$) is $\Z$-bigraded. It 
has a Maslov grading $\mu$ and an Alexander grading $A$, as defined 
in \cite{OS-knot}. Multiplication by $\la$ and $\law$ changes these 
gradings by 
\[\mu(\la^a\law^b\x)=\mu(\x)-2a\quad\text{and} 
\quad A(\la^a\law^b\x)=A(\x)-a+b\]

Subsequently, we may write
\[\CFT(K)=\bigoplus_{d,s\in\Z}\CFT_d(K,s),\]
where $d$ and $s$ denote the Maslov and Alexander grading, 
respectively. For instance, for the unknot we obtain
\[\CFT(\text{Unknot})=\F[\la,\law]=\bigoplus_{s\in\Z}\Ring_0(s),\ \ \ 
\text{where}\ \Ring_0(s)=\langle\la^a\law^b\ |\ b-a=s\rangle.\]
\begin{prop}\label{prop:shifts} 
Both $\fmap^+$ and $\fmap^-$ are homogeneous maps. If $K'$ is 
obtained from $K$ by a negative crossing change then $\fmap^+$ and 
$\fmap^-$ have bidegree $(\mu,A)=(0,0)$ and $(0,1)$, respectively. 
Similarly, if $K'$ is obtained from $K$ by a positive crossing change 
then  $\fmap^+$ and $\fmap^-$ have bidegree $(0,1)$ and $(0,0)$, 
respectively.
\end{prop}
\begin{proof}
Suppose $K'$ is obtained from $K$ by a negative crossing change. In 
the situation of Lemma~\ref{lem:reversing-arcs}, the chain maps 
$\gmap^+$ and $\gmap^-$ are homogeneous, \cite[Lemma 7.8]{AE-2}, and 
it follows from \cite[Lemma 7.2]{Ian-gading} that $\gmap^+$ and 
$\gmap^-$ are homogeneous of bidegree $(0,1/2)$. 
Furthermore, considering the bigradings, 
\[\CFT(K,p_1,p_2)\otimes_{\Abb'} \Abb=\CFT(K)\otimes_{\Abb} V\]
where $V$ is a free $\Abb$-module with two generators in bidegrees 
$(0,0)$ and $(-1,-1)$. In addition, 
\[\CFT(L,p_1,p_2)\otimes \Abb=\CFT(K')\otimes_{\Abb} W,\]
where $W=\CF(H,p_1,p_2)$ and $H$ is the right-handed Hopf link. 
Specifically, it is a bigraded chain complex of free modules with the 
generators $y_1, y_2, y_3,y_4$ in gradings
\[
\begin{split}
&(\mu(y_3),A(y_3))=(-\frac{3}{2},-\frac{3}{2}),
\quad(\mu(y_4),A(y_4))=(\frac{1}{2},\frac{1}{2})\\
&\text{and}\quad (\mu(y_1),A(y_1))=(\mu(y_2),A(y_2))
=(-\frac{1}{2},-\frac{1}{2}).
\end{split}
\]
and the differential $d$ defined by $d(y_2)=\law y_3+\la y_4$ and 
$d(y_1)=d(y_3)=d(y_4)=0$. Therefore, $\fmap^+$ preserves both Maslov 
and Alexander gradings, while $\fmap^-$ has bidegree $(0,1)$. 
The proof for a positive crossing change is analogous.
\end{proof}

Since the crossing change chain maps $\fmap^+$ and $\fmap^-$ do 
not change the Maslov index, we will drop it from the notation in 
the rest of the paper. Moreover, by degree of a homogeneous chain map
$f$, denoted by $\deg(f)$, we mean the Alexander grading degree of 
$f$.

\section{The depth of a knot and bounding the unknotting number}\label{sec:depth}
Let $K$ and $K'$ be knots in $S^3$ and $\Ibb$ denote a sequence of 
crossing changes which modifies $K$ to $K'$. We denote the length of 
$\Ibb$ by $|\Ibb|$, and the number of positive (resp. negative) 
crossing changes in $\Ibb$ by $m^{+}(\Ibb)$ (resp. $m^{-}(\Ibb)$). 
For $\bullet\in\{+,-\}$, let $u^\bullet(K,K')$ denote the minimum of 
$m^\bullet(\Ibb)$ over all such sequences $\Ibb$ of crossing changes. 
Futher, the {\emph{Gordian distance}} $u(K,K')$ between $K$ and 
$K'$ is defined as the minimum number of crossing changes required 
for modifying $K$ to $K'$. Therefore, 
\[u(K,K')\ge u^-(K,K')+u^+(K,K').\]
Define $u^\bullet(K)=u^\bullet(K,U)$, where $U$ denotes the unknot. 
Note that it is 
possible that $u^+(K)$ and/or $u^-(K)$ are realized in 
an unknotting sequence which does not have minimal length.
The knot $K'$ is called \emph{Gordian adjacent} to $K$ if there 
exists a minimal unknotting sequence for $K$ containing $K'$.  
Equivalently, the Gordian distance $u(K,K')$ from $K$ to $K'$
is $u(K)-u(K')$. Based on Theorem~\ref{thm:cobordisim-maps} we 
make the following definition.
\begin{defn}\label{def:crossings}
Given the knots $K,K'\subset S^3$,
consider all pairs of homogeneous chain maps
\begin{align*}
&\fmap^+:\CFT(K)\ra \CFT(K') \ \ \text{and}\ \ \ 
\fmap^-:\CFT(K')\ra \CFT(K)
\end{align*}
of degrees $m^+=\deg(\fmap^+)$ and $m^-=\deg(\fmap^-)$ 
such that $\fmap^-\circ\fmap^+$ and $\fmap^+\circ\fmap^-$ are chain 
homotopic to multiplication by $\law^m$, where $m^-+m^+=m$. 
Define $\depth^-(K,K'),\depth^+(K,K')$ and $\depth(K,K')$  
as the least values for the integers $\deg(\fmap^-),\deg(\fmap^+)$ 
and $m=\deg(\fmap^-)+\deg(\fmap^+)$ (respectively) among all such 
pairs. In particular, define $\depth^\pm(K)=\depth^\pm(K,U)$
and $\depth(K)=\depth(K,U)$, where $U$ denotes the unknot.
\end{defn}

When $K'=U$, the chain complex $\CFT(U)$ is chain homotopic to $\Abb$ 
(with trivial differentials). For defining $\depth^\pm(K)$ and 
$\depth(K)$, we are thus lead to consider all pairs of homogeneous 
chain maps
\begin{align*}
&\fmap^+:\CFT(K)\ra \Abb \ \ \text{and}\ \ \ 
\fmap^-:\Abb\ra \CFT(K)
\end{align*}
of degrees $m^+=\deg(\fmap^+)$ and $m^-=\deg(\fmap^-)$ 
such that $\fmap^-\circ\fmap^+$ is multiplication by 
$\law^{m}$  and $\fmap^+\circ\fmap^-$ is chain homotopic 
to multiplication by $\law^{m}$.
The discussion of the previous section, and in particular 
Theorem~\ref{thm:cobordisim-maps} and 
Proposition~\ref{prop:shifts}, imply the following theorem.

\begin{thm}\label{thm:unknotting-number}
Given a pair of knots $K,K'\subset S^3$, 
$u^\bullet(K,K')$ is bounded below by $\depth^\bullet(K,K')$
for $\bullet\in\{-,+\}$, while $u(K,K')$  is bounded below by 
$\depth(K,K')$. 
\end{thm}

\begin{remark}\label{rmk:depthmirror}
Given the knots $K$ and $K'$ in $S^3$, any pair of chain maps 
$\fmap^+$ and $\fmap^-$ satisfying the assumptions of 
Definition~\ref{def:crossings} would induce chain maps 
\[\ovl{\fmap}^-:\CFT(- {K'})\to\CFT(-{K})\quad\text{and}
\quad\ovl{\fmap}^+:\CFT(-{K})\to\CFT(-{K'})\]
of degrees $m^+$ and $m^-$, respectively. Moreover, 
$\ovl{\fmap}^-\circ\ovl{\fmap}^+\simeq \law^m$ and 
$\ovl{\fmap}^+\circ\ovl{\fmap}^-\simeq \law^m$. Thus, 
\[\depth^-(-{K},-{K'})=\depth^+(K,K'),\quad
\depth^+(-{K},-{K'})=\depth^-(K,K')\quad\text{and}\quad \depth(K,K')=\depth(-K,-K').\]\end{remark}

Let us denote the homology of $\CFT(K,s)$ by 
$\Hbb(K,s)$ for every $s\in\Z$, and set 
$\Hbb(K)=\bigoplus_s \Hbb(K,s)$. Then 
$\Hbb(K)$ is a module over $\Abb=\Fbb[\la,\law]$. Let 
$\Tbb(K)$ denote the torsion submodule of $\Hbb(K)$, i.e.
\[\Tbb(K)=\{x\in \Hbb(K)\ |\ \exists a\in \Abb-\{0\}\ s.t.\ 
ax=0\}.\]
It is clear that $\Tbb(K)$ is a sub-module of $\Hbb(K)$, and 
there is a short exact sequence
\begin{align*}
\begin{diagram}
0&\rTo{}&\Tbb(K)&\rTo{\imath_K}&\Hbb(K)&\rTo{\pi_K}&
\Abb(K)&\rTo{}&0,
\end{diagram}
\end{align*} 
where $\Abb(K)$, defined by the above exact sequence, 
is the torsion-free part of 
$\Hbb(K)$. Fix a sequence $\Ibb$ of crossing changes which 
modify $K$ to the unknot. Correspondingly, we obtain the 
$\Abb$-homomorphisms $\fmap_\Ibb^+:\Hbb(K)\ra \Abb$ and 
$\fmap_\Ibb^-:\Abb\ra \Hbb(K)$. 
The map $\fmap^+_\Ibb$ induces a map 
$\fmap^+_{\Ibb,\Tbb}:\Tbb(K)\ra \Abb$, while $\fmap^-_\Ibb$
induces the map $\fmap^-:\Abb\ra \Abb(K)$.
 
\begin{lem}\label{lem:injectivity}
The map $\fmap^-:\Abb\ra \Abb(K)$ induced by 
$\fmap^-_\Ibb$ is injective, while 
the map $\fmap^+_{\Ibb,\Tbb}:\Tbb(K)\ra \Abb$ is trivial. We thus
have a map $\fmap^+:\Abb(K)\ra \Abb$ induced by $\fmap^+_\Ibb$, 
which is injective. The induced maps are homogeneous with 
respect to the Alexander grading.\end{lem}
\begin{proof}
Let $m^+=\deg(\fmap^+_\Ibb)$ and $m^-=\deg(\fmap^-_\Ibb)$.
If $x\in\Tbb(K)$ and $ax=0$ for $0\neq a\in\Abb$, it follows 
that $a\fmap^+_\Ibb(x)=0$ in $\Abb$, implying that 
$\fmap^+_{\Ibb}(x)=0$. Since the restriction $\fmap^+_{\Ibb,\Tbb}$
of $\fmap^+_{\Ibb}$ to $\Tbb(K)$ is trivial, a
map $\fmap^+:\Abb(K)\ra \Abb$ is induced by 
$\fmap_\Ibb^+$. Let us now assume that 
$x\in\Hbb(K)$ is in the kernel of $\fmap^+_\Ibb$. Then 
$\law^{m}x=\fmap^-_\Ibb\circ\fmap^+_\Ibb(x)=0$, 
implying that $x\in\Tbb(K)$. In particular, 
$\fmap^+:\Abb(K)\ra \Abb$ is injective. 
On the other hand, if $a\in \Abb$ and 
$x=\fmap^-_\Ibb(a)\in\Tbb(K)$,
it follows that $0=\fmap^+_\Ibb(x)
=\law^{m}a$, implying that 
$a=0$. Thus $\fmap^-:\Abb\ra \Abb(K)$ is injective. This completes 
the proof of the lemma, as the statement about the Alexander 
grading follows immediately from our previous discussions.
\end{proof}
\begin{prop}\label{prop:structure-of-dent}
There is a sequence $0=i_0(K)<i_1(K)<\cdots<i_{n}(K)=\nu^-(K)$
associated with every knot $K\subset S^3$, and an identification
\begin{equation}\label{eq:dent-identification}
\Abb(K)
=\left\langle\la^{i_k(K)}\law^{i_{n-k}(K)}\ |\ 
k\in\{0,1,\ldots,n\}
\right\rangle_\Abb.
\end{equation}   
Moreover, the identification of Equation~\ref{eq:dent-identification} 
preserves the Alexander grading.
\end{prop}
\begin{proof}
If we set $\law=1$ and consider $\CFT(K)$ as a chain complex filtered 
by the Alexander filtration, we obtain an identification
\begin{equation}\label{eq:identification}
\Hbb(K,s)=H_{\star}\left(C\{\max(i,j-s)\le 0\}\right).
\end{equation}
Under this identification, 
the homomorphism induced by inclusion
\[v_{s}:H_{\star}\left(C\{\max(i,j-s)\le 0\}\right)
\to H_{\star}\left (C\{\max(i,j-s-1)\le 0\}\right)\]
is equal to multiplication by $\law$. Recall that $\nu^-=\nu^-(K)$ 
is the smallest $s$ such that the map
\[h_s:H_{\star}\left(C\{\max(i,j-s)\le 0\}\right)
\to H_{\star}\left (C\{i\le 0\}\right)=\Fbb[U]\]
induced by inclusion is surjective, where $U=\la\law$.  
It is clear that for all $s$, 
under the identification of Equation \ref{eq:identification}, 
$\Tbb(K,s)$ is equal to the kernel of $h_s$ and so the restriction 
of $h_s$ to $\Abb(K,s)\cong \Fbb[U]$ is injective.  
Furthermore, for all $s\ge \nu^-$, multiplication by $\law$ is an 
isomorphism from $\Abb(K,s)$ to $\Abb(K,s+1)$.   
Let $a$ denote the generator of $\Abb(K,\nu^-)$ i.e. 
$h_{\nu^-}(a)=1$. The above observations imply that for any 
$b\in \Abb(K,s)$ 
\begin{equation}\label{eq:nontorsionelem}
\begin{cases}
\begin{array}{ll}
b=\law^{s-\nu^-}p_b(U)a&\text{if}\quad s\ge \nu^-\\
\law^{\nu^--s}b=p_b(U)a&\text{if}\quad s<\nu^-
\end{array}
\end{cases}
\end{equation}
for some polynomial $p_b\in\Fbb[U]$. If 
$b\in \Abb(K,s)$ is the generator,  $p_b(U)=1$ for 
$s\ge \nu^-$. Suppose now that $s<\nu^-$ and $b$ is the 
generator of $\Abb(K,s)$ as before. 
Since $\la^{\nu^--s}a\in \Abb(K,s)$, we have 
$\la^{\nu^--s}a=p'(U)b$, and so $U^{\nu^--s}a=p_b(U)p'(U)a$. 
Therefore, $p_b(U)p'(U)=U^{\nu^--s}$ and $p_b(U)=U^{j_s}$ for 
some $0\le j_s\le \nu^--s$. 

Additionally, $\Hbb(K)$ is symmetric under exchanging the variables 
$\la$ and $\law$, so for all $s\le -\nu^-$ multiplication by $\la$ 
is an isomorphism from $\Abb(K,s)$ to $\Abb(K,s-1)$. Let 
$a'\in\Abb(K,-\nu^-)$ denote the generator. It is straightforward that, 
\[\la^{2\nu^-}a=U^{\nu^-}a'\quad \text{and}\quad 
\law^{2\nu^-}a'=U^{\nu^-}a.\]
Further, for any generator $b\in \Abb(K,s)$ with 
$-\nu^- \le s\le \nu^-$ we have $\la^{\nu^-+s}b=U^{j'_s}a'$ 
where $j'_s=j_s+s$.  

Then, we define a grading preserving $\Abb$-module homomorphism  
\[\imath:\Abb(K)=\oplus_s\Abb(K,s)\to \Abb\]
by setting $\imath(b)=\la^{j_s}\law^{j_s'}$ for the generator 
$b\in \Abb(K,s)$. For instance, if $s\ge \nu^-$ is non-negative
then $\imath(b)=\law^{s}$, while if $s\le -\nu^-$ is non-positive
then $\imath(b)=\la^{-s}$. It is clear that $\imath$ is injective 
and it identifies $\Abb(K)$ with an ideal generated by at most 
$2\nu^-+1$ monomial of the form $\la^{i}\law^j$ with 
$0\le i,j\le \nu^-$ in $\Abb$. This set of generators contains 
a unique minimal subset 
\[\{\la^{i_k}\law^{j_k}\ |\ 0=i_0<i_1<...<i_n=\nu^-\ 
\text{and}\ \nu^-=j_0>j_1>...>j_n=0\}\]
that generates the image of $\imath$. The symmetry of 
$\Hbb(K)$ implies that $j_k=i_{n-k}$ for all $k=0,...,n$.
\end{proof}

\begin{defn}\label{defn:dent-depth}
Under the identification of Equation~\ref{eq:dent-identification},
for every knot $K\subset S^3$ the sequence 
\[\imath(K)=(0=i_0(K)<i_1(K)<\cdots<i_{n(K)}(K)=\nu^-(K))\] 
is called the {\emph{ideal sequence}} associated with the knot $K$.
The ideal $\Abb(\imath)$ associated with a sequence 
$\imath=(0=i_0<i_1<\cdots<i_n)$ is defined as
\[\Abb(\imath)
=\left\langle\la^{i_k}\law^{i_{n-k}}\ |\ 
k\in\{0,1,\ldots,n\}
\right\rangle_\Abb,\]
and we identify $\Abb(K)=\Abb(\imath(K))$. For finite increasing 
sequences $\imath,\imath'$ of non-negative integers as above
define the {\emph{distance}} $\ell(\imath,\imath')$ from $\imath$ 
to $\imath'$ as the smallest value for $p$ such that 
$\law^p\Abb(\imath')\subset \Abb(\imath)$. Given the knots 
$K,K'\subset S^3$, define the {\emph{negative distance}} 
$\ell^-(K,K')$ as $\ell(\imath(K),\imath(K'))$. Define the 
{\emph{positive distance}} by $\ell^+(K,K')=\ell^-(-{K},-{K'})$, 
where $-{K}$ denotes the mirror image of $K$.
Define the positive/negative {\emph{depth}} of a knot 
$K$ by $\ell^\pm(K)=\ell^\pm(K,U)$, where $U$ denotes the unknot.
\end{defn} 
Note that under the identification of 
Equation~\ref{eq:dent-identification},
the negative depth of $K$ is equal to $\nu^-(K)$.

\begin{prop}\label{prop:dent-depth}
Let $K$ and $K'$ be  knots in $S^3$. Then 
\[\depth^+(K,K')\geq \max\{\ell^+(K,K'),\ell^-(K',K)\}\quad\text{and}
\quad \depth^-(K,K')\geq \max\{\ell^-(K,K'),\ell^+(K',K)\}.\]  
\end{prop}
\begin{proof}
It is straightforward corollary of the definition, that 
$\depth^-(K,K')=\depth^+(K',K)$. So, remark \ref{rmk:depthmirror} 
implies that it suffices to show that 
$\depth^-(K,K')\ge \ell^-(K,K')$.  By definition, there exists 
$\Abb$-homomorphisms
\[f:\Abb(K)\to\Abb(K')\quad\text{and}\quad g:\Abb(K')\to\Abb(K)\]
such that $f\circ g$ and $g\circ f$ are equal to multiplication by 
$\law^m$, and $\deg (g)=\depth^-(K,K')$. Under the identification of 
Equation \ref{eq:dent-identification}, it is easy to check that $f$ 
and $g$ are the restriction of $\Abb$-homomorphisms from $\Abb$ to 
$\Abb$ defined by multiplication with polynomials $p$ and $q$ in 
$\Abb$. Since, $pq=\law^m$ and $\deg (g)=\depth^-(K,K')$, we have 
$g=\law^{\depth^-(K,K')}$ and so $\depth^-(K,K')\ge \ell^-(K,K')$. 
\end{proof}

Theorem \ref{thm:unknotting-number} and the above proposition imply 
that $\ell^\pm(K,K')\le u^\pm(K,K')$.

\begin{cor}\label{cor:bounds}
For any knot $K\subset S^3$, we have
\[u^-(K)\ge\depth^-(K)\ge\nu^-(K)\ge \tau(K)\quad\text{and}
\quad u^+(K)\ge\depth^+(K)\ge \nu^-(-{K})\ge -\tau(K).\]
Therefore, for $0\le t\le 1$, we have $-t\depth^-(K)\le 
\Upsilon_K(t)\le t\depth^+(K)$.
\end{cor}

\begin{proof}
The first two claims follow from Proposition \ref{prop:dent-depth} 
and the inequality $\nu^-(K)\ge\tau(K)$ from 
\cite[Proposition 2.3]{Hom-Wu:tau}. The last claim follows from the 
inequality $-t\nu^-(K)\le \Upsilon_K(t)$ from \cite[Proposition 4.7]
{OSS-upsilon}.
\end{proof}

\section{The torsion obstruction}\label{sec:torsion}
Let us assume that a sequence $\Ibb$ of crossing changes 
is used to unknot $K\subset S^3$. Let us further 
assume that $m^+=m^+(\Ibb)$ and $m^-=m^-(\Ibb)$,
while $m=m^++m^-=|\Ibb|$. 
The argument of Lemma~\ref{lem:injectivity} then implies 
that multiplication by $\law^m$ trivializes all of 
$\Tbb(K)$. This observation gives 
a weaker obstruction to the unknotting number. 
\begin{defn}
Define the {\emph{negative torsion depth}} $\Tt^-(K)$ of a knot 
$K\subset S^3$ to be the smallest integer $m$ such that 
multiplication by $\law^m$ is trivial on $\Tbb(K)$. Let 
$\Tt^+(K)=\Tt^-(-{K})$. Then $\Tt(K)=\max\{\Tt^-(K),\Tt^+(K)\}$
is called the {\emph{torsion depth}} of $K$.
\end{defn}
Consider the homomorphism 
$\widehat{\phi}:\Abb\to\Fbb[\law]$ defined by $\widehat{\phi}(\la)=0$ 
and  $\widehat{\phi}(\law)=\law$. This homomorphism makes 
$\Fbb[\law]$ into an $\Abb$-module. We define
\[\Ehat(K)=\CFT(K)\otimes_{\widehat{\phi}}\Fbb[\law]\quad\text{and}\quad
\widehat{\Hbb}(K)=H_{\star}(\Ehat(K)).\]
Then $\widehat{\Hbb}(K)$ is a $\Fbb[\law]$-module, with a 
free summand isomorphic to $\Fbb[\law]$ and a torsion summand 
denoted by $\widehat{\Tbb}(K)$. Define $\widehat{\Tt}(K)$ as
the smallest $m$ such that multiplication by $\law^m$ is trivial 
on $\widehat{\Tbb}(K)$.  The following proposition is a 
straightforward corollary of previous definitions and  
discussions.

\begin{prop}\label{prop:torsion-bound} 
For any knot $K\subset S^3$, the torsion classes
$\widehat{\Tt}(K)$, $\widehat{\Tt}(-{K})$, and $\Tt(K)$
are lower bounds for $\depth(K)$,
and thus for the unknotting number $u(K)$.
\end{prop}

\begin{prop}\label{prop:torsion-nontrivial}
If the genus $g(K)$ of a knot $K\subset S^3$ is strictly 
bigger than $\tau(K)$ then $\Tbb(K)\neq 0$, and in particular,
$\Tt^-(K)\geq 1$.
\end{prop}

\begin{proof}
The differential $d$ of the chain complex $\CFT(K)$ may be written
as $d=\sum_{i,j\geq 0}\la^i\law^jd^{i,j}$.
Using a spectral sequence determined by $(\CFT(K),d)$, 
we can replace $\CFT(K)$ with page $1$ of the 
aforementioned spectral sequence and assume that $d^{0,0}=0$.
Let $x$ denote a generator of $\widehat{\HFKT}(K,g(K))$. 
 If a generator 
$y$ appears in $d^{i,0}(x)$ (where $i>0$), it follows that 
\[g(K)=A(x)=A(\la^iy)=A(y)-i
<A(y).\] 
Since $\widehat{\HFKT}(K,s)=0$ for $s>g(K)$, the above observation
implies that $d^{i,0}(x)=0$.
In particular, $d(x)=\law^p z$ for some $p>0$ and some $z$
representing a class $[z]\in \Hbb(K)$. Clearly, $\law^p [z]=0$ in 
$\Hbb(K)$. If $z=d(x')$ for some $x'\in\CFT(K)$, then $d(x+\law^px')=0$. Since $\tau(K)<g(K)$, the image of $x+\law^px'$ under the chain map $\CFT(K)\to \widehat{\CFT}(K)$ represents a trivial homology class. Thus, $x$ appears in $d^{0,i}(y)$ (where $i>0$) for some generator $y\in\widehat{\HFKT}(K)$. So,
\[A(y)=A(\law^ix)=A(x)+i>g(K)\]
which is a 
contradiction. In particular, $[z]$ is  non-zero 
in $\Tbb(K)$. 
\end{proof}

\begin{cor}\label{cor:t-non-trivial}
If $K$ is a non-trivial knot then $\widehat{\Tt}(K)>0$ and $\Tt(K)>0$.
\end{cor}

\begin{proof}
The first claim is a trivial consequence of the definition.
Since $K$ is non-trivial, $g(K)\geq 1$.
If $\tau(K)<g(K)$, Proposition~\ref{prop:torsion-nontrivial} gives 
the second claim. Otherwise, $\tau(-K)=-g(K)<g(-K)$
and $\Tt(K)\geq \Tt^-(-K)>0$.
\end{proof}

\begin{prop}\label{prop:connected-sum-torsion}
Suppose $K$ and $K'$ are knots in $S^3$. Then, 
\[\max\{\widehat{\Tt}(K),\widehat{\Tt}(K')\}\leq 
\widehat{\Tt}(K\#K')\leq \widehat{\Tt}(K)+\widehat{\Tt}(K').\]
\end{prop}
\begin{proof}
By K\"{u}nneth theorem for homology, there is an exact sequence
\begin{displaymath}
\begin{diagram}
0&\rTo&\widehat{\Hbb}(K)\otimes\widehat{\Hbb}(K')&\rTo& 
\widehat{\Hbb}(K\#K')&\rTo &\mathrm{Tor}_{\Fbb[\law]}
(\widehat{\Hbb}(K),\widehat{\Hbb}(K'))&\rTo&0.
\end{diagram}
\end{displaymath}
Thus, $\widehat{\Hbb}(K\#K')$ has torsion summands isomorphic to 
$\widehat{\Tbb}(K)$ and $\widehat{\Tbb}(K')$ and so 
\[\widehat{\Tt}(K\#K')\ge\max\{\widehat{\Tt}(K),
\widehat{\Tt}(K')\}.\]
Moreover, multiplication by 
$\law^{\min\{\widehat{\Tt}(K),\widehat{\Tt}(K')\}}$ is trivial on 
$\widehat\Tbb(K)\otimes\widehat\Tbb(K')$ and  
$\mathrm{Tor}_{\Fbb[\law]}(\widehat{\Hbb}(K),\widehat{\Hbb}(K'))$. 
Therefore, $\widehat{\Tt}(K\#K')$ is at most 
$\widehat{\Tt}(K)+\widehat{\Tt}(K')$. 
\end{proof}
\begin{remark}
One can construct a similar lower bound $\Tt_{p/q}$ by sending 
$\la$ and $\law$ to $\lav^p$ and $\lav^q$ in $\Fbb[\lav]$, 
respectively, which satisfies in a statement
similar to Proposition~\ref{prop:connected-sum-torsion}. 
\end{remark}

\section{Examples and applications}\label{sec:examples}
\begin{example}\label{ex:torus-knot}
Let $K=T_{p,q}$ be the $(p,q)$ torus knot with $0<p<q$. The chain 
homotopy type of $\CFT(K)$ is specified by the Alexander polynomial 
of $K$ \cite{OS-lspace}. Specifically, the symmetrized Alexander 
polynomial of $K$ is equal to 
\[\Delta_{K}(t)=t^{-\frac{(p-1)(q-1)}{2}}\frac{(t^{pq}-1)(t-1)}
{(t^p-1)(t^q-1)}=\sum_{i=0}^{2n}(-1)^it^{a_i}\]
for a sequence $a_0>a_1>...>a_{2n}$ of integers where 
$a_i=-a_{2n-i}$. The complex $\CFT(K)$ is chain homotopic to the 
bigraded complex freely generated over $\Abb$ with generators 
$\{\x_i\}_{i=0}^{2n}$ and differential 
\[d\x_i=\begin{cases}
\la^{a_{i-1}-a_{i}}\x_{i-1}+\law^{a_{i}-a_{i+1}}\x_{i+1}
\quad&\text{if}\ i\ \text{is odd}\\
0\quad &\text{if}\ i\ \text{is even}.
\end{cases}\]
Furthermore, the gradings are specified by 
$\mu(\x_i)=m_i$ and $A(\x_i)=a_i$
where $m_i$ is defined recursively by $m_0=0$ and
\[m_{2i}=m_{2i-1}-1\quad\text{and}\quad
m_{2i+1}=m_{2i}-2(a_{2i}-a_{2i+1})+1.\]
Consequently, $\Tbb(K)=0$ and  $\Abb(K)=\Hbb(K)$ is generated by 
$[\x_{2i}]$ for $i=0,...,n$. Moreover, 
\[\law^{a_{2i-1}-a_{2i}}\x_{2i}=\la^{a_{2i-2}-a_{2i-1}}\x_{2i-2}.\]
Thus, $\imath(T_{p,q})=(i_0=0<i_1<...<i_{n})$ where 
\[i_k=\sum_{j=0}^{2(n-k)}(-1)^ja_j.\]
For any knot $K$, $\CFT(-{K})\simeq\CFT(K)^\star$. So for 
$-{K}=T_{p,-q}$, the above discussion implies that $\CFT(-{K})$ is 
chain homotopic to the chain complex freely generated over $\Abb$ 
with generators $\{\x_i\}_{i=0}^{2n}$ and differential 
\[
d\x_i=\begin{cases}
0\quad&\text{if}\ i\ \text{is odd}\\
\la^{a_{i-1}-a_i}\x_{i-1}+\law^{a_i-a_{i+1}}\x_{i+1}\quad
&\text{if}\ i\ \text{is even}.
\end{cases}
\]
Moreover, the bigradings of generators is given by 
$(\mu(\x_i),A(\x_i))=(-m_{2n-i},a_i)$. Thus, $\Abb(-{K})\cong \Abb$ 
is generated by $\sum_{k=0}^n\la^{i_{n-k}}\law^{i_k}\x_{2k}$, 
while $[\x_{2k+1}]$ is torsion of order $i_{k+1}$ 
for $k=0,...,n-1$. Therefore, 
\[\Tt(T_{p,q})=\Tt^+(T_{p,q})=i_n=\nu^-(T_{p,q})
=\frac{(p-1)(q-1)}{2}.\]

Consider $\Ehat(K)=\CFT(K)\otimes_{\hat{\phi}}\Fbb[\law]$, 
where as before $\widehat{\phi}:\Abb\to\Fbb[\law]$ is the homomorphisms 
defined by $\widehat{\phi}(\la)=0$ and $\widehat{\phi}(\law)=\law$.  By the 
above discussion, $\widehat{\Hbb}(K)$ has a free summand generated by 
$[\x_0]$. Moreover, for each $1\le i\le n$, $[\x_{2i}]$ is a torsion 
class of order $a_{2i-1}-a_{2i}$. It is easy to check that 
\[a_1-a_2=p-1\quad\text{and}\quad a_{2i-1}-a_{2i}\le p-1\quad
\text{for\ every\ }i=1,...,n.\] 
Therefore, $\widehat{\Tt}(T_{p,q})=p-1$.

{\bf Special case: $p=2$, $q=2n+1$.} For the torus knot 
$T_{2,2n+1}$ we have
\[\Delta_{T_{2,2n+1}}(t)=\sum_{i=0}^{2n}(-1)^{i}t^{i-n}.\]
So $a_i=n-i$ for $0\le i\le 2n$, and thus 
\[\imath(T_{2,2n+1})=(0<1<2<...<n)\quad\text{and}
\quad\Abb(T_{2,2n+1})=\langle \la^i\law^j\ |\ i+j\ge n\rangle_\Abb.\]

{\bf Special case: $p=3$, $q=3k\pm 1$.} Suppose $q=3k+1$. First, 
we compute the symmetrized Alexander polynomial of $T_{3,3k+1}$:

\[\begin{split}
\Delta_{T_{3,3k+1}}(t)&=t^{-3k}\frac{(t^{3(3k+1)}-1)(t-1)}
{(t^{3k+1}-1)(t^3-1)}=t^{-3k}\frac{t^{2(3k+1)}+t^{3k+1}+1}{t^2+t+1}\\
&=t^{-3k}\frac{t^{3k+2}(t^{3k}-1)+t^{3k}(t^2+t+1)+1-t^{3k}}
{t^2+t+1}\\ &=\sum_{i=1}^{k} (t^{3i}-t^{3i-1})+1+
\sum_{i=-k}^{-1}(t^{3i}-t^{3i+1}).
\end{split}\]
Therefore, $n=2k$, and 
\[i_{j}=\begin{cases}
j\quad&\text{if}\ 0\le j< k\\
2j-k\quad&\text{if}\ k\le j\le n.
\end{cases}\quad\Rightarrow\quad
\Abb(T_{3,3k+1})=\langle\la^{i}\law^j\ |\ 2i+j\ge 3k\ 
\text{and}\ i+2j\ge 3k\rangle_\Abb.\]
For $q=3k-1$, an analogous argument implies that 
\[\Abb(T_{3,3k-1})=\langle \la^i\law^j\ |\ 2i+j\ge 3k-2\ 
\text{and}\ i+2j\ge 3k-2\rangle_\Abb\]

More generally, the ideal sequence for the torus knot $T_{p,pn+1}$ takes the 
explicit form
\begin{equation}\label{eq:T-ideal-sequence-1}
\begin{split}
&\imath(T_{p,pn+1})=\left(0<1<\cdots<n<n+2<\cdots<3n<3n+3<\cdots
<{p \choose 2}n\right)\\ &\text{or equivalently,}\quad
i_k=\left(k-\frac{n}{2}\left\lfloor\frac{k}{n}\right\rfloor\right)
\left(\left\lfloor\frac{k}{n}\right\rfloor+1\right),
\quad \text{for }k=0,1,\ldots,n(p-1).
\end{split}
\end{equation}
One useful computation  is the degree computation
for the generator
\[\la^{i_{\lfloor n(p-1)/2\rfloor}}\law^{i_{\lceil n(p-1)/2\rceil}}
\in\Abb(T_{p,pn+1}),\]
which follows from Equation~\ref{eq:T-ideal-sequence-1}:
\begin{equation}\label{eq:md-torus-knot}
\begin{split}
i_{\lfloor n(p-1)/2\rfloor}+
i_{\lceil n(p-1)/2\rceil}
&=\left(\left\lfloor \frac{n(p-1)}{2} \right\rfloor
-\frac{n}{2}\left\lfloor \frac{\left\lfloor \frac{n(p-1)}{2} 
\right\rfloor}{n} \right\rfloor\right)\left(
\left\lfloor \frac{\left\lfloor \frac{n(p-1)}{2} \right\rfloor}{n} 
\right\rfloor+1\right)\\
&\quad\quad +\left(\left\lceil \frac{n(p-1)}{2} \right\rceil
-\frac{n}{2}\left\lfloor \frac{\left\lceil \frac{n(p-1)}{2} 
\right\rceil}{n} \right\rfloor\right)\left(
\left\lfloor \frac{\left\lceil \frac{n(p-1)}{2} \right\rceil}{n} 
\right\rfloor+1\right)\\
&=n\left\lfloor\frac{p^2}{4}\right\rfloor.
\end{split}
\end{equation}
The minimum degree of a monomial in $\Abb(K)$ will be denoted
by $\md(K)$. The above computation shows that 
$\md(T_{p,pn+1})=n\left\lfloor\frac{p^2}{4}\right\rfloor$.

\end{example}

\begin{remark}\label{remark:A-for-T}
One can in fact show that for every $p<q$, there is an inclusion
\begin{equation}\label{eq:A-for-T}
\Abb(T_{p,q})\leq \Abb_{p,q}=
\left\langle \la^i\law^j\ \Big|\ ki+(p-k)j\geq 
\frac{k(p-k)(q-1)}{2}\quad\text{for}\ k=1,\ldots,p-1 
\right\rangle_\Abb.
\end{equation}
However, the equality is not satisfied for $p>3$, although the 
two ideals are very closely related. 
\end{remark}

\begin{prop}\label{prop:Gordianadj}
If a torus knot $K=T_{{p,p'}}$ with 
$0<p<p'$ is Gordian adjacent to a torus knot $K'=T_{q,q'}$ with 
$0<q<q'$, then 
\[\Abb(T_{q,q'})\le\Abb(T_{p,p'})\quad\text{and}
\quad\law^{u}\Abb(T_{p,p'})\le\Abb(T_{q,q'}),\] where 
$u=u(K')-u(K)=\frac{(p-1)(p'-1)}{2}-\frac{(q-1)(q'-1)}{2}$. 
In particular, $\mathfrak{a}(T_{q,q'})\ge \mathfrak{a}(T_{p,p'})$.
\end{prop}

\begin{proof} Since $T_{p,p'}$ is Gordian adjacent to $T_{q,q'}$, 
there exists $\Abb$-homomorphisms 
\[f:\Hbb(T_{p,p'})\to \Hbb(T_{q,q'})\quad\text{and}\quad 
g:\Hbb(T_{q,q'})\to\Hbb(T_{p,p'})\]
such that $f\circ g=g\circ f=\law ^{u}$. Note that 
$\Hbb(T_{p,p'})=\Abb(T_{p,p'})$ and $\Hbb(T_{q,q'})=\Abb(T_{q,q'})$. 
So, $f$ and $g$ are defined by multiplication with polynomials 
$p,q\in\Abb=\Fbb[\la,\law]$. Thus, 
$f\circ g=g\circ f=\law ^{u}$ implies that 
$f=\law^{m^+}$ and $g=\law^{m^-}$ such that 
$m^++m^-=u$. On the other hand, by 
Corollary~\ref{cor:bounds}, a minimal unknotting 
sequence for a torus knot only consists of negative crossing changes. 
Thus $\deg f=m^+=0$ and $\deg g=m^-=u$. Therefore, 
$f=\mathrm{id}$, $g$ is multiplication by $\law^u$ and 
$\Abb(T_{q,q'})\leq\Abb(T_{p,p'})$ and 
$\law^{u}\Abb(T_{p,p'})\leq\Abb(T_{q,q'})$.  
\end{proof}

The computations in Example~\ref{ex:torus-knot} and the Proposition~\ref{prop:Gordianadj} have a number 
of quick consequences. One outcome is the following corollary 
that was suggested to us by  Jennifer Hom. This result was first 
proved by Borodzik and Livingston in \cite{Bor-Liv}.
\begin{cor} If a torus knot $T_{{p,p'}}$ with 
$0<p<p'$ is Gordian adjacent to a torus knot $T_{q,q'}$ with 
$0<q<q'$, then $p\le q$.
\end{cor}
\begin{proof}
Assume that \[\imath(T_{p,p'})=(i_0<\cdots<i_n)\quad \text{and}\quad
\imath(T_{q,q'})=(j_0<\cdots<j_m).\]
Proposition~\ref{prop:Gordianadj} implies that 
$\law^{j_m-i_n}\Abb(T_{p,p'})\le \Abb(T_{q,q'})$. Thus,
\[\law^{j_m-i_n+i_{n-1}}\la^{i_1}
=\law^{u(K')-u(K)}\law^{i_{n-1}}\la^{i_1}\in \Abb(K').\]
Since $i_1=j_1=1$, $i_n-i_{n-1}=p-1$ and $j_m-j_{m-1}=q-1$, 
the above conclusion implies 
\[j_m-i_n+i_{n-1}\geq j_{m-1}\quad \Leftrightarrow\quad
\frac{(q-1)(q'-1)}{2}-p+1\geq \frac{(q-1)(q'-1)}{2}-q+1
\quad \Leftrightarrow\quad q\geq p,
\]
completing the proof.
\end{proof}
We also obtain a proof of the following corollary. The second 
statement of the corollary was first proved by Peter Feller 
\cite{Feller}.

\begin{cor}\label{cor:adjacency-T}
If the torus knot $T_{p,pn+1}$ is Gordian adjacent to the torus knot 
$T_{q,qm+1}$ then 
\[n\left\lfloor\frac{p^2}{4}\right\rfloor\leq 
m\left\lfloor\frac{q^2}{4}\right\rfloor.
\]
If $T_{2,n}$ is Gordian adjacent to $T_{3,m}$, where $n$ is odd
and $m$ is not a multiple of $3$, then 
$n\le \frac{4}{3}m+\frac{1}{3}$. 
\end{cor}
\begin{proof}
Proposition~\ref{prop:Gordianadj} implies that $\md(T_{p,pn+1})\le\md(T_{q,qm+1})$. So, following the computations of 
Example~\ref{ex:torus-knot} we have
\[n\left\lfloor\frac{p^2}{4}\right\rfloor\leq 
m\left\lfloor\frac{q^2}{4}\right\rfloor.
\]
Moreover, from the same example we know that 
$\Abb(T_{3,m})\le\Abb(T_{2,n})$ if and only if for any pair $(i,j)$ 
such that $i+2j\ge m-1$ and $j+2i\ge m-1$, we have 
$i+j\ge \frac{n-1}{2}$. It is clear that
\[\min\{i+j\ |\ i+2j\ge m-1\ \text{and}\ 2i+j\ge m-1\}=\left\lceil
\frac{2(m-1)}{3}\right\rceil=\left\lfloor\frac{2m}{3}-\frac{1}{3}
\right\rfloor.\]
Thus, $\frac{n-1}{2}\le \frac{2m}{3}-\frac{1}{3}$ and 
$n\le \frac{4}{3}m+\frac{1}{3}$.
\end{proof}

\begin{example}
An interesting example is the case of the figure $8$ knot,
where the chain complex is generated by $5$ generators 
$X,Y,Z,W$ and $B$, where $d(B)=d(X)=0$ while
$d(W)=\la Z+\law Y$, $d(Y)=\la Y$ and $d(Z)=\law X$.
Thus, $\Tbb(K)$ is generated by $\x=[X]$ and $\la\x$ and 
$\law\x$ are both zero. Moreover, $\Abb(K)$ is generated 
by $[B]$ and is isomorphic with $\Abb$. In particular, 
$\nu^-(K)=\nu^-(-K)=0$, while $\Tt(K)=\widehat{\Tt}(K)=1$. 
The sub-complex generated by $X,Y,Z$ and $W$ will 
be referred to as a {\emph{square}}.
\end{example}

\begin{example}\label{ex:alternating}
Alternating knots are known to have simple knot 
Floer chain complexes.
The restriction on the Alexander and Maslov grading of generators 
(that their difference is a constant number) implies that the chain 
complex decomposes as the (shifted) direct sum of a copy of 
$\CFT(\pm T_{2,2n+1})$ and several squares. In particular, if $K$ is 
an alternating knot with $\tau(K)>0$ then 
\[\imath(K)=(0<1<2<\cdots<\tau(K)),\]
while $\Tt^-(K)\leq 1$ and $\depth(K)=\depth^-(K)=\Tt^+(K)=\tau(K)$.  
\end{example}

Example~\ref{ex:alternating} gives interesting bounds on the 
{\emph{alternation number}} $\alt(K)$ of a knot $K$, defined as 
the minimum Gordian distance between $K$ and an alternating knot.
The first bound is very similar to, yet different from, 
the bound constructed in \cite[Corollary 2.2]{Feller-etal}. 

\begin{prop}\label{prop:alternation-for-K}
The alternation number $\alt(K)$ of a knot $K\subset S^3$
satisfies 
\[\alt(K)\geq \nu^-(K)-\md(K),\quad\alt(K)\geq \widehat{\Tt}(K)-1
\quad\text{and}\quad
\alt(K)\geq\min\{\Tt(K)-1,\nu^-(K)\}.\]
\end{prop}
\begin{proof}
Let us assume that $K$ is modified
to an alternating knot $K'$ using a sequence of $m^+$ positive 
crossing changes and $m^-$ negative crossing changes and that
$\alt(K)=m^++m^-$.
It follows that $\nu^-(K')\geq \nu^-(K)-m^-$.
Since $\law^{m^+}\Abb(K)$ is a subset of $\Abb(K')$, it 
follows that $\Abb(K')$ includes a monomial of degree
$m^++\md(K)$.
Nevertheless, every monomial in $\Abb(K')$ has degree at 
least $\nu^-(K')$. This means that 
\begin{align*}
&\md(K)+m^+ \geq 
\nu^-(K')\geq \nu^-(K)-m^-\quad
\Rightarrow\quad
m^++m^-\geq \nu^-(K)-\md(K),
\end{align*}
and completes the proof of the first inequality. The second and third
inequalities are easier. For the second equality note that in the 
above situation,
\[u(K,K')\geq \widehat{\Tt}(K)-\widehat{\Tt}(K')
=\widehat{\Tt}(K)-1.\]
For the third inequality, we have
\[\nu^-(K)\leq \nu^-(K')+m^-\quad\text{and}\quad
\Tt(K)\leq \Tt(K')+m^++m^-.\]

If $\nu(K')=0$ then $\nu^-(K)\le m^-\leq \alt(K)$. Otherwise, 
$\tau(K')=\nu(K')>0$ and
$\Tbb(K')$ can only include torsion elements 
trivialized by $\law$. In particular, $\Tt(K')=1$ and
$\alt(K)=m^++m^-\geq \Tt(K)-1$.
\end{proof}

For torus knots, we obtain the following corollary from our 
computations in Example~\ref{ex:torus-knot}.
Similar bounds may also be obtained using Upsilon invariants, 
c.f. \cite{Feller-etal} for the case $p<5$.

\begin{cor}\label{cor:alternation-for-T}
The alternation number of the torus knot $T_{p,pn+1}$
is at least $n\left\lfloor \frac{(p-1)^2}{4}\right\rfloor$.
\end{cor}
\begin{proof}
Using the first inequality in 
Proposition~\ref{prop:alternation-for-K} we have
\begin{align*}
\alt(T_{p,pn+1})&\geq  \nu^(T_{p,pn+1})-\md(T_{p,pn+1})
=n{p\choose 2}-n\left\lfloor\frac{p^2}{4}\right\rfloor=
n\left\lfloor\frac{(p-1)^2}{4}\right\rfloor.
\end{align*}
This completes the proof.
\end{proof}

\begin{figure}
\includegraphics[scale=0.5]{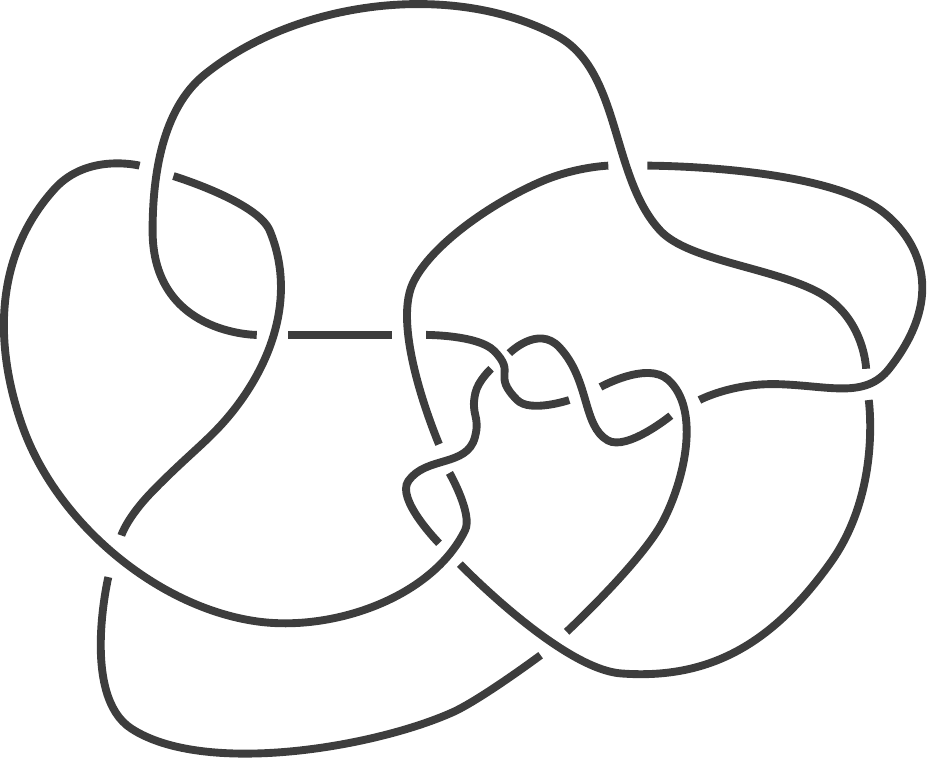}
\caption{
The knot $12n_{404}$
}\label{fig:12n404}
\end{figure}  
\begin{example}
The knot $12n_{404}$, which is a $(1,1)$ knot, is illustrated in 
 Figure~\ref{fig:12n404}. Using Rasmussen's notation 
 \cite[page 14]{Jake-oneone}, it is given by 
 the quadruple $[29,7,14,1]$. The corresponding chain complex
 $\CFT(12n_{404})$ may be computed combinatorially, e.g. using 
 Krcatovich's computer program \cite{Krcatovich}. 
 After a straight-forward change of basis, we arrive at the chain 
 complex illustrated in Figure~\ref{fig:CF-12n404}. 
 
\begin{figure}
\includegraphics[scale=0.35]{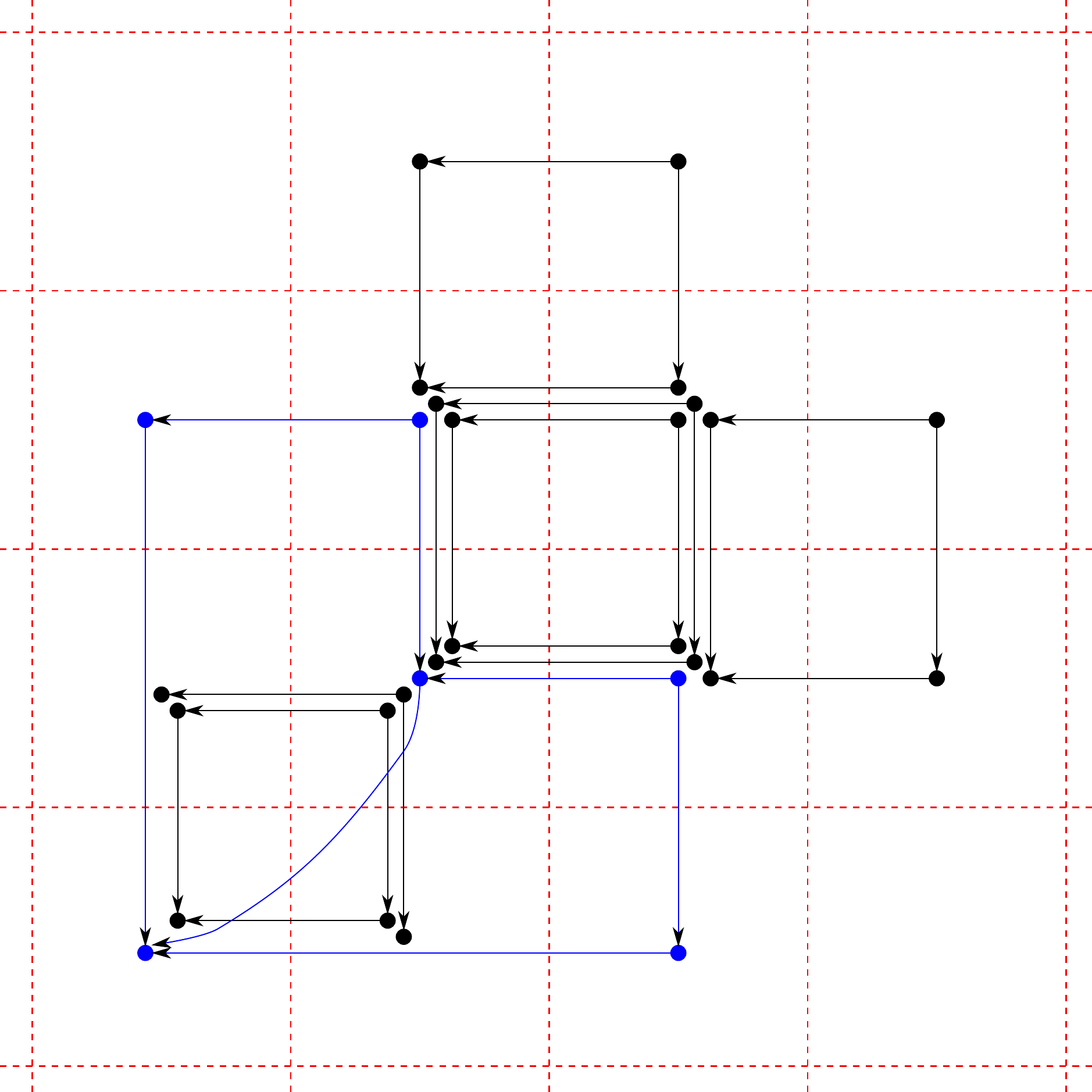}
\caption{
The chain complex associated with the knot $12n_{404}$.
}\label{fig:CF-12n404}
\end{figure}  
 
 Each dot 
 represents a generator of $\CFT (12n_{404})$. An arrow which 
 connects a dot corresponding to a generator $\x$ to a dot 
 representing a generator $\y$ and cuts $i$ vertical lines and 
 $j$ horizontal lines corresponds to the contribution of 
 $\la^i\law^j\y$ to $d(\x)$. The blue dots and the black dots 
 in the diagram  generate subcomplexes $C$ and $C'$ 
 of the knot chain complex, respectively, and we 
 obtain a decomposition  $\CFT(12n_{404})=C\oplus C'$. 
 We may then identify 
 \[C=\langle X,Y_0,Y_1,Y_2,Z_0,Z_1\rangle_\Abb,\quad
 d(Y_i)=\la^i\law^{2-i}X\quad \text{and}\quad d(Z_i)
 =\la Y_i+\law Y_{i+1}.\]
 The homology of $C$ is then generated by $\x=[X]$, with 
 $\law^2\x=\law\la\x=\la^2\x=0$. In particular, it follows that 
 $\Tt(12n_{404})\geq 2$. In fact, it is straightforward from the 
 above  presentation of chain complex to conclude that 
 $\Tt(12n_{404})=\widehat{\Tt}(12n_{404})=2$, while
 \[\depth^-(12n_{404})=\nu^-(12n_{404})=1,\quad
 \depth^+(12n_{404})=0\quad\text{and}\quad \depth(12n_{404})
 =\Tt(12n_{404})=2.\]
The knot $12n_{404}$ may be unknotted by changing $3$ 
crossings. It is not known, however, whether $u(12n_{404})$ is 
equal to $3$ or not. The alternation number $\alt(12n_{404})$ is 
$1$, which matches the lower bound given by the last two 
inequalities in Proposition~\ref{prop:alternation-for-K}.
\end{example}

\begin{example}
Consider the $(2,-1)$ cable of the torus knot $T_{2,3}$, 
which is denoted by $T_{2,3;2,-1}$. The chain complex associated 
with this knot is illustrated in Figure~\ref{fig:TCable-1}.

\begin{figure}
\def\svgwidth{5.5cm}
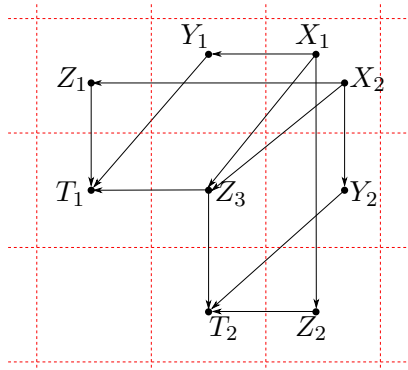
\caption{
The chain complex associated with the knot $T_{2,3;2,-1}$.
}\label{fig:TCable-1}
\end{figure} 

The chain complex is generated over $\Fbb[\la,\law]$ by the $9$ 
generator $X_1,X_2,Y_1,Y_2,Z_1,Z_2,Z_3,T_1$ and $T_2$. The 
differential is given by $d(T_i)=0$, $d(Y_i)=\la\law T_i$, for 
$i=1,2$ and  
\begin{align*}
&d(Z_1)=\law T_1,\quad d(Z_3)=\la T_1+\law T_2,\quad 
d(Z_2)=\la T_2,\\ &d(X_1)=\la Y_1+\la\law Z_3+\law^2 Z_2
\quad\text{and}\quad d(X_2)=\la^2 Z_1+\la\law Z_3+\law Y_2,
\end{align*}

The generators of homology may then be specified as 
$\tb_1=[T_1],\tb_2=[T_2], \y_1=[Y_1+\la Z_1]$ and  
$\y_2=[Y_2+\law Z_2]$, where we have 
\[\la \tb_1=\law \tb_2,\quad \law \tb_1=\la \tb_2=0
\quad\text{and}\quad\la \y_1=\law \y_2.\]
It thus follows that 
\[\Hbb(T_{2,3;2,-1})=\Abb(T_{2,3;2,-1})\oplus \Tbb(T_{2,3;2,-1})
=\langle \la,\law\rangle_\Abb\oplus 
\frac{\langle \la,\law\rangle_\Abb}
{\langle \la^2,\law^2 \rangle_\Abb}.\]
In particular, $\Tt(T_{2,3;2,-1})=\widehat\Tt(T_{2,3;2,-1})=2$, 
$\nu^-(T_{2,3;2,-1})=1$ and $\depth^-(T_{2,3;2,-1})=2$. Since 
the torsion invariant $\Tt(T_{2,3})$ is zero, 
it follows that the Gordian distance between 
$T_{2,3;2,-1}$ and the trefoil $T_{2,3}$ is at least $2$.
\end{example}

\begin{example}
Let us now consider the $(2,-3)$ cable of the torus knot $T_{2,3}$,
which is denoted by $T_{2,3;2,-3}$. We focus on the mirror image
$K=-T_{2,3;2,-3}$ of the aforementioned knot. The chain complex 
associated with $K$ is illustrated in Figure~\ref{fig:TCable-2}.

\begin{figure}
\def\svgwidth{6.5cm}
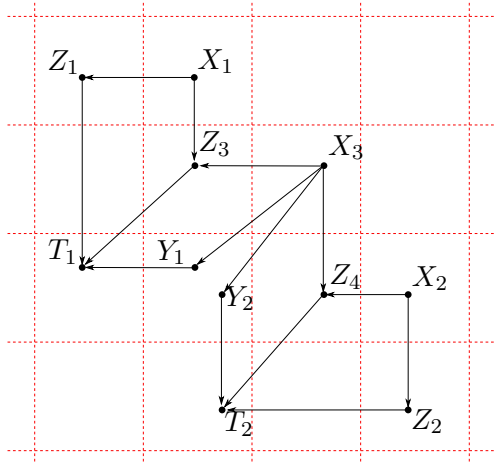
\caption{
The chain complex associated with the knot $-T_{2,3;2,-3}$.
}\label{fig:TCable-2}
\end{figure}

The chain complex is generated over $\Fbb[\la,\law]$ by $11$ 
generators $T_1,T_2,X_1,X_2,X_3,Y_1,Y_2,Z_1,Z_2,Z_3$ and $Z_4$.
The differential is given by $d(T_1)=d(T_2)=0$ and 
\begin{align*}
&d(Y_1)=\la T_1,\quad d(Y_2)=\law T_2,\quad, d(Z_1)=\law^2 T_1,
\quad d(Z_2)=\la^2 T_2,\quad d(Z_3)=\la\law T_1,\quad 
d(Z_4)=\la\law T_2,\\
&d(X_1)=\la Z_1+\law Z_3,\quad d(X_2)=\law Z_2+\la Z_4\quad
\text{and}\quad d(X_3)=\la Z_3+\law Z_4+\la\law(Y_1+Y_2).
\end{align*}
The homology of the above chain complex is generated by
 $\tb_1=[T_1],\tb_2=[T_2], \y_1=[Z_3+\law Y_1]$ 
and $\y_2=[Z_4+\la Y_2]$, while we also have
\[\la \tb_1=\law^2 \tb_1=\law \tb_2=\la^2 \tb_2=0\quad\text{and}
\quad \la \y_1=\law \y_2.\]
It thus follows that  
\[\Hbb(-T_{2,3;2,-3})=\Abb(-T_{2,3;2,-3})\oplus \Tbb(-T_{2,3;2,-3})
=\langle \la,\law\rangle_\Abb\oplus\left(
\frac{\Abb}{\langle \la,\law^2\rangle_\Abb}\oplus
\frac{\Abb}{\langle\la^2,\law\rangle_\Abb}\right).\]
By considering the dual complex, one can show that 
\[\Hbb(T_{2,3;2,-3})=\Abb\oplus
\frac{\langle \la,\law\rangle_\Abb}
{\langle \la^2,\law^2 \rangle_\Abb}.\]
In particular, we have $\nu^-(-T_{2,3;2,-3})=1$ and 
$\nu^-(T(2,3;2,-3))=0$ while the torsion invariants are 
non-trivial: 
\[\Tt^-(T_{2,3;2,-3})=\Tt^+(T(2,3;2,-3))=\widehat\Tt(T_{2,3;2,-3})
=\widehat\Tt(-T_{2,3;2,-3})=2.\]
\end{example}

\begin{example}
This example illustrates that $\Hbb(K)$ is not necessarily the 
direct sum of $\Abb(K)$ and $\Tbb(K)$. Let 
$K=T_{4,5}\#-T_{2,3;2,5}\#T_{2,3}$. The chain complex for $K$ 
is large, with many acyclic pieces. Nevertheless, it includes a 
direct summand, which we would like to study. Specifically, 
$\CFT(K)=C\oplus C'$, where  the chain complex $C$ is illustrated in 
Figure~\ref{fig:TCable-3} and the homology of $C'$ is freely generated 
by torsion elements $\tb_i$ such that $\la \tb_i=\law \tb_i=0$. 

\begin{figure}
\def\svgwidth{6.5cm}
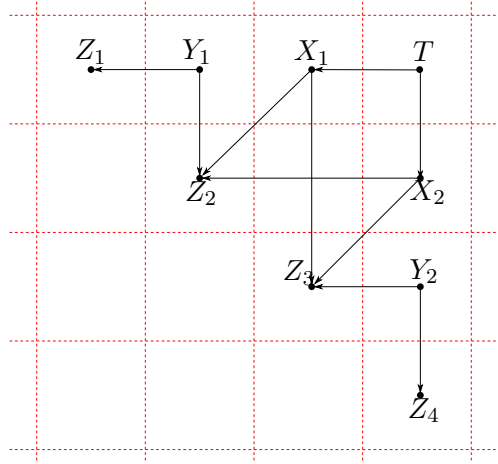
\caption{
The chain complex $C$ associated with the knot $T_{4,5}\#-T_{2,3;2,5}\#T_{2,3}$.
}\label{fig:TCable-3}
\end{figure}

The chain complex is generated over $\Fbb[\la,\law]$ by the $9$ 
generators $X_1,X_2,Y_1,Y_2,Z_1,Z_2,Z_3,Z_4$ and $T$. The 
differential is given by $d(Z_i)=0$ for $i=1,2,3,4$ and 
\begin{align*}
&d(Y_1)=\la Z_1+\law Z_2,\quad d(Y_2)=\la Z_3+\law Z_4,\quad
d(T)=\la X_1+\law X_2\\
& d(X_1)=\la\law Z_2+\law^2 Z_3\quad\quad 
\text{and}\quad\quad d(X_2)=\la^2 Z_2+\la\law Z_3.
\end{align*}
The homology of $C$ is then generated by the classes
$\z_i=[Z_i]$ for $i=1,2,3,4$, while we have
\[\la \z_1=\law \z_2,\quad \la \z_3=\law \z_4,\quad
\la\law \z_2=\law^2 \z_3\quad 
\text{and}\quad \la^2 \z_2=\la\law \z_3.\]
In particular, $\tb=\la \z_2-\law \z_3$ is a torsion element, and 
$\la \tb=\law \tb=0$.
We then have a short exact sequence 
\begin{displaymath}
\begin{diagram}
0&\rTo&\frac{\Abb}{\langle \la,\law\rangle_\Abb}&\rTo &
H_*(C)&\rTo &\Abb(K)=
\langle \la^3,\la^2\law,\la\law^2,\law^3\rangle_\Abb&
\rTo&0,
\end{diagram}
\end{displaymath}
which does not split. The chain complex $C$ is an illustration of 
pieces which may appear in a knot chain complex and make the homology
and the unknotting invariants interesting. The next {\emph{virtual}}
example gives another instance of this phenomenon.
\end{example}

\begin{example}
Let $C=C_{i,j}$ denote the chain complex generated over $\Abb$ 
by the generators $X_1,X_2,Y_1,Y_2$ and $Z$ with 
\[A(X_1)=-A(X_2)=i,\quad A(Y_1)=-A(Y_2)=j
\quad\text{and}\quad A(Z)=0.\]

 \begin{figure}
\def\svgwidth{5.5cm}
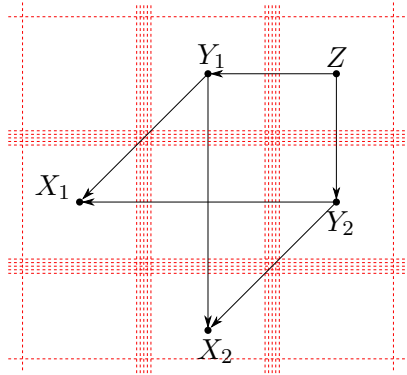
\caption{
The chain complex $C_{i,j}$.
}\label{fig:Cij}
\end{figure} 

The differential $d=d_{i,j}$ of $C$ is defined by setting
 $d(X_1)=d(X_2)=0$ and
\begin{align*}
&d(Y_1)= \la^{i}\law^{j}X_1+\law^{i+j}X_2,&&
d(Y_2)=\la^{i+j}X_1+\la^{j}\law^{i}X_2,&&
d(Z)=\la^jY_1+\law^jY_2.
\end{align*} 
Figure~\ref{fig:Cij} illustrates this chain complex. We treat
$C_{i,j}$ as a direct summand in a knot chain complex, or a 
{\emph{virtual}} knot chain complex.

It is then not hard to check that the homology group 
$\Hbb=\Hbb_{i,j}$ of $C$ is generated by the homology classes 
$\x_1=[X_1]$ and $\x_2=[X_2]$. Furthermore, $\tb=\la^ix_1+\law^ix_2$ 
is a torsion element in $\Hbb$. In fact, 
\[\law^{j}\tb=[dY_1]=0\quad\text{and}\quad\la^{j}\tb=[dY_2]=0.\]
Let us now assume that $\fmap^-:\Hbb\ra \Abb$ and 
$\fmap^+:\Abb\ra \Hbb$ are homogeneous maps of degrees $m^-$ and 
$m^+$, respectively. It follows that $\fmap^-(\x_1)=\law^{m^-+i}$ and 
$\fmap^-(\x_2)=\law^{m^-}\la^{i}$, while 
$\fmap^+(1)=\law^{m^+-i}\x_1$. But this implies that 
\[\law^{m^++m^-}\x_2=\fmap^+(\fmap^-(\x_2))=\law^{m^++m^--i}\la^i\x_1
\quad\Rightarrow\quad \law^{m^++m^--i}\tb=0.\]
In particular, $m^++m^--i\geq j$, or $m^++m^-\geq i+j$.
In other words, $\depth(C)\geq i+j$. It is then easy to conclude 
that $\depth(C)=i+j$, while $\depth^-(C)=\nu^-(C)=i$ and 
$\depth^+(C)=0$. Moreover, $\Tt(C)=j$. Thus, $C=C_{i,j}$  
gives an example with $\depth(C)=\nu^-(C)+\Tt(C)$. It is interesting
to note that in this example, 
$\widehat{\Tt}(C_{ij})=i+j$. 
\end{example}

\bibliographystyle{hamsalpha}
\bibliography{HFBibliography}
\end{document}

%% file: commands.tex
\newcommand{\Ibb}{\mathbb{I}}
\newcommand{\Fbb}{\mathbb{F}}
\newcommand{\Abb}{\mathbb{A}}
\newcommand{\Tbb}{\mathbb{T}}
\newcommand{\depth}{\mathfrak{l}}

\newcommand{\md}{\mathfrak{a}}
\newcommand{\alt}{\mathrm{alt}}

\newcommand{\lav}{\mathsf{v}}
\newcommand{\law}{\mathsf{w}}

\newcommand{\F}{\mathbb F}

\newcommand{\fmap}{\mathfrak{f}}
\newcommand{\gmap}{\mathfrak{g}}

\newcommand{\hmap}{\mathfrak{h}}

\newcommand{\Sqmap}{{\mathfrak{S}}}

\newcommand{\Ehat}{\widehat{\mathrm{CF}}}

\newcommand{\la}{\mathsf{u}}

\newcommand{\Hbb}{\mathbb{H}}

\newcommand{\ovl}{\overline}

\newcommand{\Tt}{\mathfrak{t}}

\newcommand\Dual{\mathcal D}
\newcommand\Duality\Dual

\newcommand{\relspinc}{{\underline{\spinc}}}

\newcommand\x{\mathbf x}
\newcommand\tb{\mathbf t}

\newcommand\z{\mathbf z}

\newcommand\y{\mathbf y}

\newcommand\ModSphere{\ModFlow\left({\mathbb S}\longrightarrow
\Sym^{g-1}(\Sigma_{1})\times \Sym^2(\Sigma_{2})\right)}
\newcommand\ModSpheres\ModSphere
\newcommand\CF{CF}

\newcommand\UnparModSp{\widehat \ModSp}
\newcommand\UnparModFlow\UnparModSp

\newcommand{\spinc}{\mathfrak s}

\newcommand\ModMaps{\mathcal M}
\newcommand\ModSp\ModMaps

\newcommand\Tb{{\mathbb T}_{\beta}}
\newcommand\Tc{{\mathbb T}_{\gamma}}
\newcommand\Td{{\mathbb T}_{\delta}}

\newcommand\alphas{\mbox{\boldmath$\alpha$}}

\newcommand\betas{\mbox{\boldmath$\beta$}}
\newcommand\gammas{\mbox{\boldmath$\gamma$}}

\newcommand\deltas{\mbox{\boldmath$\delta$}}

\newcommand\Ring{\mathbb A}

\newcommand\spincrel\relspinc

\newtheorem{thm}{Theorem}[section]
\newtheorem{prop}[thm]{Proposition}
\newtheorem{cor}[thm]{Corollary}

\newtheorem{lem}[thm]{Lemma}

\newtheorem{remark}[thm]{Remark}

\newtheorem{quest}{Question}[section]

\theoremstyle{definition}
\newtheorem{defn}{Definition}[section]
\newtheorem{example}{Example}[section]

\def\endproof{\relax\ifmmode\expandafter\endproofmath\else
  \unskip\nobreak\hfil\penalty50\hskip.75em\hbox{}\nobreak\hfil\bull
  {\parfillskip=0pt \finalhyphendemerits=0 \bigbreak}\fi}
\def\endproofmath$${\eqno\bull$$\bigbreak}
\def\bull{\vbox{\hrule\hbox{\vrule\kern3pt\vbox{\kern6pt}\kern3pt\vrule}\hrule}}

\newcommand{\Z}{\mathbb{Z}}

\newcommand{\ModSWfour}{\mathcal{M}}
\newcommand{\ModFlow}{\ModSWfour}

\newcommand{\SpinC}{{\mathrm{Spin}}^c}

\newcommand\abuts\Rightarrow
\newcommand\Sym{\mathrm{Sym}}

\newcommand{\Farc}{{\mathbb{I}}}

\newcommand{\HFKT}{\text{HFK}}
\newcommand{\CFT}{\mathrm{CF}}

\newcommand{\ra}{\rightarrow}
\newcommand{\Sig}{\Sigma}

\makeatletter
\providecommand\@dotsep{5}
\def\listtodoname{List of Todos}
\def\listoftodos{\@starttoc{tdo}\listtodoname}
\makeatother

%% file: resolution-3.pdf_tex
\begingroup%
  \makeatletter%
  \providecommand\color[2][]{%
    \errmessage{(Inkscape) Color is used for the text in Inkscape, but the package 'color.sty' is not loaded}%
    \renewcommand\color[2][]{}%
  }%
  \providecommand\transparent[1]{%
    \errmessage{(Inkscape) Transparency is used (non-zero) for the text in Inkscape, but the package 'transparent.sty' is not loaded}%
    \renewcommand\transparent[1]{}%
  }%
  \providecommand\rotatebox[2]{#2}%
  \ifx\svgwidth\undefined%
    \setlength{\unitlength}{560.87132034bp}%
    \ifx\svgscale\undefined%
      \relax%
    \else%
      \setlength{\unitlength}{\unitlength * \real{\svgscale}}%
    \fi%
  \else%
    \setlength{\unitlength}{\svgwidth}%
  \fi%
  \global\let\svgwidth\undefined%
  \global\let\svgscale\undefined%
  \makeatother%
  \begin{picture}(1,0.53485683)%
    \put(0,0){\includegraphics[width=\unitlength,page=1]{resolution-3.pdf}}%
    \put(0.23867193,0.12644215){\color[rgb]{0,0,0}\makebox(0,0)[lb]{\smash{modification}}}%
    \put(0,0){\includegraphics[width=\unitlength,page=2]{resolution-3.pdf}}%
    \put(0.66858338,0.1812676){\color[rgb]{0,0,0}\makebox(0,0)[lb]{\smash{band }}}%
    \put(0.65587993,0.12644215){\color[rgb]{0,0,0}\makebox(0,0)[lb]{\smash{surgery}}}%
    \put(0,0){\includegraphics[width=\unitlength,page=3]{resolution-3.pdf}}%
    \put(0.23867193,0.42062727){\color[rgb]{0,0,0}\makebox(0,0)[lb]{\smash{modification}}}%
    \put(0,0){\includegraphics[width=\unitlength,page=4]{resolution-3.pdf}}%
    \put(0.66858338,0.47545273){\color[rgb]{0,0,0}\makebox(0,0)[lb]{\smash{band }}}%
    \put(0.65587993,0.42062732){\color[rgb]{0,0,0}\makebox(0,0)[lb]{\smash{surgery}}}%
    \put(0,0){\includegraphics[width=\unitlength,page=5]{resolution-3.pdf}}%
  \end{picture}%
\endgroup%

%% file: resolution-4.pdf_tex
\begingroup%
  \makeatletter%
  \providecommand\color[2][]{%
    \errmessage{(Inkscape) Color is used for the text in Inkscape, but the package 'color.sty' is not loaded}%
    \renewcommand\color[2][]{}%
  }%
  \providecommand\transparent[1]{%
    \errmessage{(Inkscape) Transparency is used (non-zero) for the text in Inkscape, but the package 'transparent.sty' is not loaded}%
    \renewcommand\transparent[1]{}%
  }%
  \providecommand\rotatebox[2]{#2}%
  \ifx\svgwidth\undefined%
    \setlength{\unitlength}{145.17946384bp}%
    \ifx\svgscale\undefined%
      \relax%
    \else%
      \setlength{\unitlength}{\unitlength * \real{\svgscale}}%
    \fi%
  \else%
    \setlength{\unitlength}{\svgwidth}%
  \fi%
  \global\let\svgwidth\undefined%
  \global\let\svgscale\undefined%
  \makeatother%
  \begin{picture}(1,0.68890406)%
    \put(0,0){\includegraphics[width=\unitlength]{resolution-4.pdf}}%
    \put(0.4383184,0.63866451){\color[rgb]{0,0,0}\makebox(0,0)[lb]{\smash{$\beta$}}}%
    \put(0.60756701,0.63669648){\color[rgb]{0,0,0}\makebox(0,0)[lb]{\smash{$\gamma$}}}%
    \put(0.71777544,0.63669648){\color[rgb]{0,0,0}\makebox(0,0)[lb]{\smash{$\delta$}}}%
    \put(0.56348364,0.24545657){\color[rgb]{0,0,0}\makebox(0,0)[lb]{\smash{$A$}}}%
    \put(0.67369207,0.12973773){\color[rgb]{0,0,0}\makebox(0,0)[lb]{\smash{$B$}}}%
    \put(0.53593154,0.05810225){\color[rgb]{0,0,0}\makebox(0,0)[lb]{\smash{$C$}}}%
  \end{picture}%
\endgroup%

%% file: TCable-1.pdf_tex
\begingroup%
  \makeatletter%
  \providecommand\color[2][]{%
    \errmessage{(Inkscape) Color is used for the text in Inkscape, but the package 'color.sty' is not loaded}%
    \renewcommand\color[2][]{}%
  }%
  \providecommand\transparent[1]{%
    \errmessage{(Inkscape) Transparency is used (non-zero) for the text in Inkscape, but the package 'transparent.sty' is not loaded}%
    \renewcommand\transparent[1]{}%
  }%
  \providecommand\rotatebox[2]{#2}%
  \ifx\svgwidth\undefined%
    \setlength{\unitlength}{433.42307281bp}%
    \ifx\svgscale\undefined%
      \relax%
    \else%
      \setlength{\unitlength}{\unitlength * \real{\svgscale}}%
    \fi%
  \else%
    \setlength{\unitlength}{\svgwidth}%
  \fi%
  \global\let\svgwidth\undefined%
  \global\let\svgscale\undefined%
  \makeatother%
  \begin{picture}(1,0.89981366)%
    \put(0,0){\includegraphics[width=\unitlength,page=1]{TCable-1.pdf}}%
    \put(0.69339282,0.80336738){\color[rgb]{0,0,0}\makebox(0,0)[lb]{\smash{$X_1$}}}%
    \put(0,0){\includegraphics[width=\unitlength,page=2]{TCable-1.pdf}}%
    \put(0.77003284,0.81329311){\color[rgb]{0,0,0}\makebox(0,0)[lb]{\smash{}}}%
    \put(0.82490405,0.69954273){\color[rgb]{0,0,0}\makebox(0,0)[lb]{\smash{$X_2$}}}%
    \put(0.82490405,0.42267693){\color[rgb]{0,0,0}\makebox(0,0)[lb]{\smash{$Y_2$}}}%
    \put(0.4995868,0.42267693){\color[rgb]{0,0,0}\makebox(0,0)[lb]{\smash{$Z_3$}}}%
    \put(0.41652708,0.80336738){\color[rgb]{0,0,0}\makebox(0,0)[lb]{\smash{$Y_1$}}}%
    \put(0.11543559,0.69954273){\color[rgb]{0,0,0}\makebox(0,0)[lb]{\smash{$Z_1$}}}%
    \put(0.11543559,0.42267693){\color[rgb]{0,0,0}\makebox(0,0)[lb]{\smash{$T_1$}}}%
    \put(0.48574352,0.09735969){\color[rgb]{0,0,0}\makebox(0,0)[lb]{\smash{$T_2$}}}%
    \put(0.69339282,0.09735969){\color[rgb]{0,0,0}\makebox(0,0)[lb]{\smash{$Z_2$}}}%
  \end{picture}%
\endgroup%

%% file: TCable-2.pdf_tex
\begingroup%
  \makeatletter%
  \providecommand\color[2][]{%
    \errmessage{(Inkscape) Color is used for the text in Inkscape, but the package 'color.sty' is not loaded}%
    \renewcommand\color[2][]{}%
  }%
  \providecommand\transparent[1]{%
    \errmessage{(Inkscape) Transparency is used (non-zero) for the text in Inkscape, but the package 'transparent.sty' is not loaded}%
    \renewcommand\transparent[1]{}%
  }%
  \providecommand\rotatebox[2]{#2}%
  \ifx\svgwidth\undefined%
    \setlength{\unitlength}{540bp}%
    \ifx\svgscale\undefined%
      \relax%
    \else%
      \setlength{\unitlength}{\unitlength * \real{\svgscale}}%
    \fi%
  \else%
    \setlength{\unitlength}{\svgwidth}%
  \fi%
  \global\let\svgwidth\undefined%
  \global\let\svgscale\undefined%
  \makeatother%
  \begin{picture}(1,0.94444444)%
    \put(0.39034015,0.64029168){\color[rgb]{0,0,0}\makebox(0,0)[lb]{\smash{$Z_3$}}}%
    \put(0.08175301,0.80709228){\color[rgb]{0,0,0}\makebox(0,0)[lb]{\smash{$Z_1$}}}%
    \put(0,0){\includegraphics[width=\unitlength,page=1]{TCable-2.pdf}}%
    \put(0.38878907,0.80896763){\color[rgb]{0,0,0}\makebox(0,0)[lb]{\smash{$X_1$}}}%
    \put(0,0){\includegraphics[width=\unitlength,page=2]{TCable-2.pdf}}%
    \put(0.84027778,0.875){\color[rgb]{0,0,0}\makebox(0,0)[lb]{\smash{}}}%
    \put(0.65566025,0.63312157){\color[rgb]{0,0,0}\makebox(0,0)[lb]{\smash{$X_3$}}}%
    \put(0.4453246,0.32405497){\color[rgb]{0,0,0}\makebox(0,0)[lb]{\smash{$Y_2$}}}%
    \put(0.30540642,0.41775911){\color[rgb]{0,0,0}\makebox(0,0)[lb]{\smash{$Y_1$}}}%
    \put(0.08318418,0.41945826){\color[rgb]{0,0,0}\makebox(0,0)[lb]{\smash{$T_1$}}}%
    \put(0.44345635,0.06703999){\color[rgb]{0,0,0}\makebox(0,0)[lb]{\smash{$T_2$}}}%
    \put(0.82520345,0.07287758){\color[rgb]{0,0,0}\makebox(0,0)[lb]{\smash{$Z_2$}}}%
    \put(0,0){\includegraphics[width=\unitlength,page=3]{TCable-2.pdf}}%
    \put(0.82693778,0.36512239){\color[rgb]{0,0,0}\makebox(0,0)[lb]{\smash{$X_2$}}}%
    \put(0.65819838,0.36811167){\color[rgb]{0,0,0}\makebox(0,0)[lb]{\smash{$Z_4$}}}%
  \end{picture}%
\endgroup%

%% file: TCable-3.pdf_tex
\begingroup%
  \makeatletter%
  \providecommand\color[2][]{%
    \errmessage{(Inkscape) Color is used for the text in Inkscape, but the package 'color.sty' is not loaded}%
    \renewcommand\color[2][]{}%
  }%
  \providecommand\transparent[1]{%
    \errmessage{(Inkscape) Transparency is used (non-zero) for the text in Inkscape, but the package 'transparent.sty' is not loaded}%
    \renewcommand\transparent[1]{}%
  }%
  \providecommand\rotatebox[2]{#2}%
  \ifx\svgwidth\undefined%
    \setlength{\unitlength}{540bp}%
    \ifx\svgscale\undefined%
      \relax%
    \else%
      \setlength{\unitlength}{\unitlength * \real{\svgscale}}%
    \fi%
  \else%
    \setlength{\unitlength}{\svgwidth}%
  \fi%
  \global\let\svgwidth\undefined%
  \global\let\svgscale\undefined%
  \makeatother%
  \begin{picture}(1,0.94444444)%
    \put(0.35912917,0.53638585){\color[rgb]{0,0,0}\makebox(0,0)[lb]{\smash{$Z_2$}}}%
    \put(0.13276842,0.82670593){\color[rgb]{0,0,0}\makebox(0,0)[lb]{\smash{$Z_1$}}}%
    \put(0,0){\includegraphics[width=\unitlength,page=1]{TCable-3.pdf}}%
    \put(0.58145002,0.8244005){\color[rgb]{0,0,0}\makebox(0,0)[lb]{\smash{$X_1$}}}%
    \put(0,0){\includegraphics[width=\unitlength,page=2]{TCable-3.pdf}}%
    \put(0.84027778,0.875){\color[rgb]{0,0,0}\makebox(0,0)[lb]{\smash{}}}%
    \put(0.81834382,0.53898027){\color[rgb]{0,0,0}\makebox(0,0)[lb]{\smash{$X_2$}}}%
    \put(0.8178076,0.37705832){\color[rgb]{0,0,0}\makebox(0,0)[lb]{\smash{$Y_2$}}}%
    \put(0.35325628,0.82619832){\color[rgb]{0,0,0}\makebox(0,0)[lb]{\smash{$Y_1$}}}%
    \put(0.82576769,0.8235193){\color[rgb]{0,0,0}\makebox(0,0)[lb]{\smash{$T$}}}%
    \put(0.81451522,0.09235026){\color[rgb]{0,0,0}\makebox(0,0)[lb]{\smash{$Z_4$}}}%
    \put(0,0){\includegraphics[width=\unitlength,page=3]{TCable-3.pdf}}%
    \put(0.55923462,0.374337){\color[rgb]{0,0,0}\makebox(0,0)[lb]{\smash{$Z_3$}}}%
  \end{picture}%
\endgroup%

%% file: Cij.pdf_tex
\begingroup%
  \makeatletter%
  \providecommand\color[2][]{%
    \errmessage{(Inkscape) Color is used for the text in Inkscape, but the package 'color.sty' is not loaded}%
    \renewcommand\color[2][]{}%
  }%
  \providecommand\transparent[1]{%
    \errmessage{(Inkscape) Transparency is used (non-zero) for the text in Inkscape, but the package 'transparent.sty' is not loaded}%
    \renewcommand\transparent[1]{}%
  }%
  \providecommand\rotatebox[2]{#2}%
  \ifx\svgwidth\undefined%
    \setlength{\unitlength}{464bp}%
    \ifx\svgscale\undefined%
      \relax%
    \else%
      \setlength{\unitlength}{\unitlength * \real{\svgscale}}%
    \fi%
  \else%
    \setlength{\unitlength}{\svgwidth}%
  \fi%
  \global\let\svgwidth\undefined%
  \global\let\svgscale\undefined%
  \makeatother%
  \begin{picture}(1,0.93103448)%
    \put(0,0){\includegraphics[width=\unitlength,page=1]{Cij.pdf}}%
    \put(0.78448276,0.75862069){\color[rgb]{0,0,0}\makebox(0,0)[lb]{\smash{$Z$}}}%
    \put(0.47413793,0.76551724){\color[rgb]{0,0,0}\makebox(0,0)[lb]{\smash{$Y_1$}}}%
    \put(0.08103448,0.44827586){\color[rgb]{0,0,0}\makebox(0,0)[lb]{\smash{$X_1$}}}%
    \put(0.78448276,0.36206896){\color[rgb]{0,0,0}\makebox(0,0)[lb]{\smash{$Y_2$}}}%
    \put(0.47413793,0.05172413){\color[rgb]{0,0,0}\makebox(0,0)[lb]{\smash{$X_2$}}}%
  \end{picture}%
\endgroup%